\documentclass[reqno]{amsart}

\usepackage{a4wide}
\usepackage{color}
\usepackage{mathrsfs}
\usepackage{mathtools}
\usepackage{amsmath}
\usepackage{amssymb}
\usepackage{nicefrac}
\numberwithin{equation}{section}
\usepackage[colorlinks,citecolor=green,linkcolor=red]{hyperref}

\usepackage[latin1]{inputenc}

\newcommand{\B}{\mathbb{B}}
\renewcommand{\H}{{\rm H}}
\newcommand{\M}{{\rm M}}
\newcommand{\N}{\mathbb{N}}
\newcommand{\R}{\mathbb{R}}
\newcommand{\T}{{\rm T}}
\newcommand{\LIP}{{\rm LIP}}
\newcommand{\Lip}{{\rm Lip}}
\newcommand{\lip}{{\rm lip}}
\newcommand{\la}{\langle}
\newcommand{\ra}{\rangle}
\newcommand{\sfd}{{\sf d}}
\renewcommand{\d}{{\mathrm d}}
\newcommand{\eps}{\varepsilon}
\newcommand{\nchi}{{\raise.3ex\hbox{\(\chi\)}}}
\newcommand{\fr}{\penalty-20\null\hfill\(\blacksquare\)}
\newcommand{\X}{{\rm X}}

\newcommand{\mm}{\mathfrak m}

\newcommand{\n}{\boldsymbol{\sf n}}
\newcommand{\E}{{\rm E}}
\newcommand{\ppi}{\boldsymbol\pi}
\renewcommand{\b}{\boldsymbol b}

\newtheorem{theorem}{Theorem}[section]
\newtheorem{corollary}[theorem]{Corollary}
\newtheorem{lemma}[theorem]{Lemma}
\newtheorem{proposition}[theorem]{Proposition}
\newtheorem{definition}[theorem]{Definition}

\newtheorem{remark}[theorem]{Remark}

\title{Universal infinitesimal Hilbertianity of sub-Riemannian manifolds}
\author{Enrico Le Donne, Danka Lu\v{c}i\'{c}, and Enrico Pasqualetto}

\address{Dipartimento di Matematica\\
         Universit\`{a} di Pisa\\
         Largo Bruno Pontecorvo, 5\\
         56127 Pisa\\
         Italy\\
         \textit{and}\\
         University of Jyvaskyla\\
         Department of Mathematics and Statistics \\
         P.O. Box 35 (MaD) \\
         FI-40014 University of Jyvaskyla \\
         Finland}         
\email{ledonne@msri.org}

\address{University of Jyvaskyla\\
         Department of Mathematics and Statistics \\
         P.O. Box 35 (MaD) \\
         FI-40014 University of Jyvaskyla \\
         Finland}
\email{danka.d.lucic@jyu.fi}

\address{University of Jyvaskyla\\
         Department of Mathematics and Statistics \\
         P.O. Box 35 (MaD) \\
         FI-40014 University of Jyvaskyla \\
         Finland}
\email{enrico.e.pasqualetto@jyu.fi}

\begin{document}
\date{\today} 
\begin{abstract}
We prove that sub-Riemannian manifolds are infinitesimally
Hilbertian (i.e., the associated Sobolev space is Hilbert)
when equipped with an arbitrary Radon measure. The result follows
from an embedding of metric derivations into the space of
square-integrable sections of the horizontal bundle, which we
obtain on all weighted sub-Finsler manifolds. As an intermediate
tool, of independent interest, we show that any sub-Finsler distance
can be monotonically approximated from below by Finsler ones.
All the results are obtained in the general setting
of possibly rank-varying structures.
\end{abstract}
\maketitle
\tableofcontents
\section{Introduction}
\paragraph{\it General overview}\ \
In the last two decades, weakly differentiable functions over metric measure
spaces have been extensively studied and have played a fundamental role in
the development of abstract calculus in the nonsmooth setting (see, e.g.,
\cite{HKST15,Gigli12,Gigli14}).
The definition of Sobolev space we adopt in this paper is the one
introduced in \cite{DM14}, which is equivalent to the notions proposed in
\cite{Cheeger00,Shanmugalingam00,AmbrosioGigliSavare11}.
At this level of generality, however, Sobolev calculus might not be fully 
satisfactory from a functional-analytic viewpoint. For instance, not only
the Sobolev space can fail to be Hilbert (consider the Euclidean space endowed
with the \(L^\infty\)-norm and the Lebesgue measure), but it can be also
non-reflexive (as shown in \cite[Proposition 7.8]{ACM14}).
In view of this, the class of \emph{infinitesimally Hilbertian} metric measure
spaces (i.e., whose associated Sobolev space is Hilbert) is particularly relevant.
These spaces enjoy nice features, among which the strong density of
boundedly-supported Lipschitz functions in the Sobolev space (as proven
in \cite{AmbrosioGigliSavare11-3}); we refer to the introduction of
\cite{LP19} for an account of the several  benefits of working
within this class of spaces.

A strictly related concept is that of \emph{universally infinitesimally Hilbertian}
metric space, that is to say, a metric space that is infinitesimally
Hilbertian with respect to whichever Radon measure. The interest in this
property is mainly motivated by the study of metric structures that are
important from a geometric perspective, but do not carry any `canonical'
measure (such as sub-Riemannian manifolds that are not equiregular).
The purpose of this paper is to prove the following claim:
\[
\text{All sub-Riemannian manifolds are universally infinitesimally Hilbertian.}
\]
The goal will be achieved by building an isometric embedding of the
`analytic' space of \emph{derivations} over any weighted sub-Finsler manifold
(which provide us with a synthetic notion of vector field, linked to the Sobolev calculus)
into the `geometric' space of sections of the horizontal bundle.
The abstract differential structure of the space under consideration and the
behaviour of its (purely metric) tangent spaces are -- a priori -- unrelated,
thus the role of the above-mentioned embedding result is to bridge this gap, 
showing that Sobolev functions are suitable to capture the fiberwise Hilbertianity
of the horizontal bundle. As an intermediate tool, of independent interest,
we prove that a sub-Finsler distance can be monotonically approximated from
below by Finsler distances.
\bigskip
\paragraph{\it Outline of the paper}\ \
We consider a \emph{(generalised) sub-Finsler manifold} \((\M,\E,\sigma,\psi)\).
This means that \(\M\) is a smooth connected manifold, while \(\E\) is a
smooth vector bundle over \(\M\) equipped with a \emph{continuous metric}
\(\sigma\colon\E\to[0,+\infty)\) (as in Definition \ref{def:cont_metric}) and
\(\psi\colon\E\to\T\M\) is a bundle morphism; moreover, a H\"{o}rmander-like
condition is required to hold, cf.\ Definition \ref{def:sF_manifold}.
Whenever it holds that for every \(x\in\M\) the norm \(\sigma|_{\E_x}\)
on the fiber \(\E_x\) is induced by a scalar product that smoothly depends
on \(x\), we say that
\((\M,\E,\sigma,\psi)\) is a \emph{(generalised) sub-Riemannian manifold}.
This notion of sub-Riemannian manifold is the most general
one that we have in the literature (see, e.g., \cite{AgrBarBos19}).

The \emph{horizontal bundle} \(\H\M\) is obtained by `patching together'
the horizontal fibers \(\mathcal D_x\coloneqq\psi(\E_x)\), which form a
\emph{continuous distribution} on \(\M\) (in the sense of Theorem
\ref{thm:rank_vary_distr}). We then define a \emph{generalised metric}
\(\rho\colon\T\M\to[0,+\infty]\) over the tangent bundle (cf.\ Definition
\ref{def:gen_metric} for this term) as
\[
\rho(x,v)=\|v\|_x\coloneqq
\inf\big\{\sigma(u)\;\big|\;u\in\E_x,\,(x,v)=\psi(u)\big\},
\quad\text{ for every }(x,v)\in\T\M.
\]
Observe that the finiteness domain of \(\rho(x,\cdot)\) coincides with the horizontal
fiber \(\mathcal D_x\) for every \(x\in\M\). The space \(\M\)
can be made into a metric space by considering the
\emph{Carnot--Carath\'{e}odory distance}: given any two points \(x,y\in\M\),
we define \(\sfd_{\rm CC}(x,y)\) as the length of the shortest path among all
horizontal curves (i.e., tangent to \(\H\M\)) joining \(x\) and \(y\).
Here, the length of a horizontal curve is computed with respect to the generalised
metric \(\rho\). See Definition \ref{def:CC_distance} for the details.

Let us now fix a non-negative Radon measure \(\mu\)
on \((\M,\sfd_{\rm CC})\), say that \(\mu\) is finite (for simplicity).
We may consider two (completely different in nature) notions of vector field
over \((\M,\sfd_{\rm CC},\mu)\):
\begin{itemize}
\item The space \({\rm Der}^{2,2}(\M;\mu)\) of \(L^2\)-\emph{derivations} having
divergence in \(L^2\) (in the sense of \cite{DM14}). These are linear functionals
acting on Lipschitz functions and taking values into the space of Borel functions
over \(\M\), that satisfy a suitable Leibniz rule and a locality
property. The Sobolev space \(W^{1,2}(\M,\sfd_{\rm CC},\mu)\) is then obtained
in duality with \({\rm Der}^{2,2}(\M;\mu)\), as described in Definition
\ref{def:Sobolev}. The whole Section \ref{s:der} is devoted to the key
results about \(L^2\)-derivations.
\item The space \(L^2(\H\M;\mu)\) of \(2\)-integrable sections of the
horizontal bundle; see Definition \ref{def:sect_HM}. Whenever \(\M\) is a
sub-Riemannian manifold, the elements of \(L^2(\H\M;\mu)\) satisfy a
pointwise parallelogram rule (thanks to geometric considerations, see
Remark \ref{rmk:geom_pr}). Nevertheless, it is not clear -- a priori --
how to deduce from this information that the metric measure space
\((\M,\sfd_{\rm CC},\mu)\) is infinitesimally Hilbertian.
\end{itemize}

The main result of the paper aims at providing a relation between
\({\rm Der}^{2,2}(\M;\mu)\) and \(L^2(\H\M;\mu)\):
the former space is isometrically embeddable into the latter one.
The precise statement is:
\begin{theorem}[Embedding theorem]\label{thm:embedding_intro}
Let \((\M,\E,\sigma,\psi)\) be a sub-Finsler manifold with \(\sfd_{\rm CC}\) complete.
Let \(\mu\) be a finite, non-negative Borel measure on \((\M,\sfd_{\rm CC})\).
Then there exists a unique linear operator
\({\rm I}\colon{\rm Der}^{2,2}(\M;\mu)\to L^2(\H\M;\mu)\) such that
\[
\d_\H f(x)\big[{\rm I}(\b)(x)\big]=\b(f)(x)
\quad\text{ holds for }\mu\text{-a.e.\ }x\in\M,
\]
for every \(\b\in{\rm Der}^{2,2}(\M;\mu)\) and \(f\in C^1_c(\M)\cap\LIP(\M)\).
Moreover, the operator \(\rm I\) satisfies
\[
{\big\|{\rm I}(\b)(x)\big\|}_x=|\b|(x)\quad\text{ for }\mu\text{-a.e.\ }x\in\M,
\]
for every \(\b\in{\rm Der}^{2,2}(\M;\mu)\).
\end{theorem}
As a consequence, sub-Riemannian
manifolds are universally infinitesimally Hilbertian:
\begin{theorem}[Infinitesimal Hilbertianity of sub-Riemannian manifolds]
\label{thm:uiH_intro}
Let \((\M,\E,\sigma,\psi)\) be a sub-Riemannian manifold with \(\sfd_{\rm CC}\)
complete. Let \(\mu\) be a non-negative Radon measure on \((\M,\sfd_{\rm CC})\).
Then the metric measure space \((\M,\sfd_{\rm CC},\mu)\) is infinitesimally
Hilbertian.
\end{theorem}
The proof of the embedding result (Theorem \ref{thm:embedding_intro})
builds upon the following key ingredients:
\begin{itemize}
\item[\(\rm a)\)] The Carnot--Carath\'{e}odory distance \(\sfd_{\rm CC}\) can be written
as pointwise limit of an increasing sequence of Finsler distances; cf.\ Theorem
\ref{thm:approx_sR_with_Riem}. This property follows from the results we develop
in Section \ref{s:approx_metr}, where we show that the sub-Finsler metric \(\rho\)
(or, more generally, any generalised metric as in Definition \ref{def:gen_metric})
can be approximated from below by Finsler ones. This technical statement can be
achieved by exploiting the lower semicontinuity of \(\rho\),
as done in Lemma \ref{lem:main_lemma}.
\item[\(\rm b)\)] The pointwise norm of a given derivation can be recovered by
just considering its evaluation at smooth \(1\)-Lipschitz functions. More precisely,
we can find a sequence \((f_n)_n\subseteq C^1_c(\M)\) of \(1\)-Lipschitz functions (with respect to \(\sfd_{\rm CC}\)) such that the identity
\(|\b|=\sup_n\b(f_n)\) holds \(\mu\)-a.e.\ for every \(\b\in{\rm Der}^{2,2}(\M;\mu)\).
This representation formula is obtained by combining item a) above with an
approximation result for Finsler manifolds proven in \cite{LP19}.
\item[\(\rm c)\)] Any derivation \(\b\in{\rm Der}^{2,2}(\M;\mu)\) can be represented
by a suitable measure \(\ppi\) on the space of continuous curves in \(\M\), as
granted by the metric version \cite{PaoSte12} of Smirnov's superposition principle
for normal \(1\)-currents; see Theorem \ref{thm:superposition_principle}. The presence
of such a measure \(\ppi\) is an essential tool in the construction of the
embedding map \({\rm I}\colon{\rm Der}^{2,2}(\M;\mu)\to L^2(\H\M;\mu)\),
which preserves the pointwise norm of all vector fields as a consequence of item b).
\end{itemize}
\medskip
\paragraph{\it Comparison with previous works}\ \
The results of the present paper enrich the list of metric
spaces that are known to be universally infnitesimally Hilbertian,
which previously consisted of:
\begin{itemize}
\item[\(\rm i)\)] Euclidean spaces \cite{GP16-2},
\item[\(\rm ii)\)] Riemannian manifolds \cite{LP19},
\item[\(\rm iii)\)] Carnot groups \cite{LP19},
\item[\(\rm iv)\)] locally \({\rm CAT}(\kappa)\)-spaces \cite{DMGPS18}.
\end{itemize}
Let us now briefly comment on the main differences and analogies
between the technique we exploit here and the previous approaches.
To the best of our knowledge, the strategy proposed in \cite{GP16-2,LP19}
does not carry over to the framework of sub-Riemannian manifolds.
In the classes of spaces i), ii), and iii), a fact which plays a
fundamental role is that any Lipschitz function (with respect to
the relevant distance) can be approximated by smooth ones having
the same Lipschitz constant; it seems that this property, achieved
by a convolution argument, cannot be generalised to sub-Riemannian
manifolds, the problem being to keep the Lipschitz constant under control.

However, a different approach has been developed in
\cite{DMGPS18} in order to overcome the lack of smoothness of
the spaces in iv). The proof in the sub-Riemannian case is
inspired by the ideas introduced in \cite{DMGPS18}: indeed,
the universal infinitesimal Hilbertianity of \(\rm CAT\) spaces
stems -- similarly to what described above -- from an embedding
result, which in turn relies upon Smirnov's superposition principle
and a representation formula for the pointwise norm of derivations.
While the former is available on any metric measure space, the latter
requires an ad hoc argument for the sub-Riemannian setting.
This makes a significant difference with \cite{DMGPS18}: on \(\rm CAT\)
spaces, distance functions from given points are \(1\)-Lipschitz and
everywhere have some form of differentiability, thus they are suitable candidates for
the representation formula; on sub-Riemannian manifolds, on the contrary,
this is no longer true, whence we need to find an alternative way to show
that there is plenty of smooth \(1\)-Lipschitz functions that are
\(\mu\)-a.e.\ differentiable (where \(\mu\) is an arbitrary measure).
Most of the present paper is actually dedicated to addressing this last point.
Once the representation formula is at disposal, the proof of the embedding result
closely follows along the lines of \cite[Theorem 6.2]{DMGPS18}.

\medskip
\noindent{\bf Acknowledgements.}
E.L.D.\ was partially supported by the Academy of Finland (grant 288501
`\emph{Geometry of subRiemannian groups}' and by grant 322898
`\emph{Sub-Riemannian Geometry via Metric-geometry and Lie-group Theory}')
and by the European Research Council
(ERC Starting Grant 713998 GeoMeG `\emph{Geometry of Metric Groups}').
D.L.\ and E.P.\ were partially supported by the Academy of Finland,
projects 274372, 307333, 312488, and 314789.

The authors would like to thank Tapio Rajala for the fruitful discussions about
the results of Section \ref{s:approx_metr}.
\section{Derivations and Sobolev calculus on metric measure spaces}\label{s:der}
We recall here the notions of derivation and Sobolev space that have been
proposed by S.\ Di Marino in \cite{DM14}.
For our purposes, a \emph{metric measure space} is a triple \((\X,\sfd,\mm)\),
where \((\X,\sfd)\) is a complete and separable metric space, while \(\mm\geq 0\)
is a locally finite Borel measure on \((\X,\sfd)\).
\medskip

We call \(\LIP(\X)\) or \(\LIP^\sfd(\X)\) the space of real-valued
Lipschitz functions on \((\X,\sfd)\), while \(\LIP_{bs}(\X)\)
or \(\LIP^\sfd_{bs}(\X)\) stand for the set of elements of \(\LIP(\X)\)
with bounded support. The (global) Lipschitz constant of \(f\in\LIP(\X)\)
is denoted by \(\Lip(f)\) or \(\Lip^\sfd(f)\), while the functions
\(\lip(f)\colon\X\to[0,+\infty)\) and \(\lip_a(f)\colon\X\to[0,+\infty)\)
are defined as
\[
\lip(f)(x)\coloneqq\varlimsup_{y\to x}\frac{\big|f(y)-f(x)\big|}{\sfd(y,x)},
\quad\lip_a(f)(x)\coloneqq\inf_{r>0}\Lip\big(f|_{B_r(x)}\big)
\]
whenever \(x\in\X\) is an accumulation point, and \(\lip(f)(x)=\lip_a(f)(x)\coloneqq 0\)
elsewhere. We say that \(\lip(f)\) and \(\lip_a(f)\) are the
\emph{local Lipschitz constant} and the \emph{asymptotic Lipschitz constant}
of the function \(f\), respectively.
The vector space of all (equivalence classes up to \(\mm\)-a.e.\ equality of)
real-valued Borel functions on \(\X\) is denoted by \(L^0(\mm)\).
\bigskip

A \emph{derivation} on \((\X,\sfd,\mm)\) is a linear map
\(\b\colon\LIP_{bs}(\X)\to L^0(\mm)\) with these two properties:
\begin{itemize}
\item[\(\rm a)\)] \textsc{Leibniz rule.} The identity \(\b(fg)=f\,\b(g)+g\,\b(f)\)
holds for every \(f,g\in\LIP_{bs}(\X)\).
\item[\(\rm b)\)] \textsc{Weak locality.} There exists a function \(G\in L^0(\mm)\)
such that \(\big|\b(f)\big|\leq G\,\lip_a(f)\) is satisfied in the
\(\mm\)-a.e.\ sense for every \(f\in\LIP_{bs}(\X)\).
\end{itemize}
The \emph{pointwise norm}
\(|\b|\coloneqq{\rm ess\,sup}\big\{\b(f)\,\big|
\,f\in\LIP_{bs}(\X),\,\Lip(f)\leq 1\big\}\)
is the minimal function (in the \(\mm\)-a.e.\ sense) that can be chosen
as \(G\) in item b) above.
\begin{definition}[The space \({\rm Der}^{2,2}(\X;\mm)\)]
Let \((\X,\sfd,\mm)\) be a metric measure space. Then
we denote by \({\rm Der}^{2,2}(\X;\mm)\) the space of all derivations
\(\b\) on \((\X,\sfd,\mm)\) such that \(|\b|\in L^2(\mm)\) and whose
distributional divergence can be represented as a function in \(L^2(\mm)\),
i.e., there exists a (uniquely determined) function \({\rm div}(\b)\in L^2(\mm)\) such that
\[\int\b(f)\,\d\mm=-\int f\,{\rm div}(\b)\,\d\mm\quad\text{ for every }f\in\LIP_{bs}(\X).\]
\end{definition}

The space \({\rm Der}^{2,2}(\X;\mm)\) is a module over the commutative
ring \(\LIP_{bs}(\X)\) and is a Banach space when endowed with the norm
\({\rm Der}^{2,2}(\X;\mm)\ni\b\mapsto\|\b\|_{2,2}\coloneqq
\big(\int|\b|^2\,\d\mm+\int{\rm div}(\b)^2\,\d\mm\big)^{\nicefrac{1}{2}}\).
\begin{lemma}\label{lem:der_continuity}
Let \((\X,\sfd,\mm)\) be a metric measure space and \(\b\in{\rm Der}^{2,2}(\X;\mm)\).
Let \((f_n)_n\subseteq\LIP_{bs}(\X)\) be a sequence with \(\sup_n\Lip(f_n)<+\infty\)
that pointwise converges to some limit \(f\in\LIP_{bs}(\X)\). Then
\[\int\varphi\,\b(f_n)\,\d\mm\longrightarrow\int\varphi\,\b(f)\,\d\mm
\quad\text{ for every }\varphi\in\LIP_{bs}(\X).\]
\end{lemma}
\begin{proof}
See item (1) of \cite[Lemma 5.4]{DMGPS18}.
\end{proof}
It will be convenient to work with the following representation formula
for the pointwise norm of the elements of \({\rm Der}^{2,2}(\X;\mm)\).
\begin{proposition}\label{prop:ptwse_norm_b}
Let \((\X,\sfd,\mm)\) be a metric measure space. Let \(\b\in{\rm Der}^{2,2}(\X;\mm)\)
be given. Fix a countable dense set \((x_k)_k\subseteq\X\).
For any \(j,k\in\N\), let \(\eta_{jk}\colon\X\to[0,1-1/j]\) be a
boundedly-supported Lipschitz function such that \(\eta_{jk}=1-1/j\)
on \(B_j(x_k)\) and \(\Lip(\eta_{jk})\leq 1/j^2\).
Then it holds that
\begin{equation}\label{eq:ptwse_norm_b}
|\b|=\underset{j,k\in\N}{\rm ess\,sup\,}
\b\big((\sfd(\cdot,x_k)\wedge j)\,\eta_{jk}\big)
\quad\text{ in the }\mm\text{-a.e.\ sense.}
\end{equation}
\end{proposition}
\begin{proof}
It follows from \cite[Proposition 5.5]{DMGPS18}.
\end{proof}

By duality with \({\rm Der}^{2,2}(\X;\mm)\), it is possible to introduce a
notion of Sobolev space \(W^{1,2}(\X,\sfd,\mm)\).
\begin{definition}[Sobolev space]\label{def:Sobolev}
Let \((\X,\sfd,\mm)\) be a metric measure space. Then we say that a function
\(f\in L^2(\mm)\) belongs to the \emph{Sobolev space} \(W^{1,2}(\X,\sfd,\mm)\)
provided there exists a continuous morphism
\(L_f\colon{\rm Der}^{2,2}(\X;\mm)\to L^1(\mm)\) of \(\LIP_{bs}(\X)\)-modules such that
\[\int L_f(\b)\,\d\mm=-\int f\,{\rm div}(\b)\,\d\mm
\quad\text{ for every }\b\in{\rm Der}^{2,2}(\X;\mm).\]
\end{definition}

The map \(L_f\) is uniquely determined. Furthermore, there exists a
function \(G\in L^2(\mm)\) such that
\[\big|L_f(\b)\big|\leq G\,|\b|\;\;\;\mm\text{-a.e.}
\quad\text{ for every }\b\in{\rm Der}^{2,2}(\X;\mm).\]
The minimal such function \(G\) (in the \(\mm\)-a.e.\ sense) is called
\emph{\(2\)-weak gradient} of \(f\) and is denoted by \(|Df|\) or \(|Df|_{\mm}\).
Then \(W^{1,2}(\X,\sfd,\mm)\) is a Banach space when equipped with the norm
\[{\|f\|}_{W^{1,2}(\X,\sfd,\mm)}\coloneqq
\left(\int|f|^2\,\d\mm+\int|Df|^2\,\d\mm\right)^{\nicefrac{1}{2}}
\quad\text{ for every }f\in W^{1,2}(\X,\sfd,\mm).\]
\begin{remark}\label{rmk:mathscr_L_f}{\rm
Let \((\X,\sfd,\mm)\) be a metric measure space.
Consider the (not necessarily complete) norm
\({\rm Der}^{2,2}(\X;\mm)\ni\b\mapsto\|\b\|_2
\coloneqq\big(\int|\b|^2\,\d\mm\big)^{\nicefrac{1}{2}}\).
Call \(\B\) the dual space of \(\big({\rm Der}^{2,2}(\X;\mm),\|\cdot\|_2\big)\).
Given any function \(f\in W^{1,2}(\X,\sfd,\mm)\), we define the element
\(\mathscr L_f\in\B\) as
\begin{equation}\label{eq:def_mathscr_L_f}
\mathscr L_f(\b)\coloneqq\int L_f(\b)\,\d\mm
\quad\text{ for every }\b\in{\rm Der}^{2,2}(\X;\mm).
\end{equation}
Then it holds that the map \(W^{1,2}(\X,\sfd,\mm)\ni f\mapsto\mathscr L_f\in\B\)
is linear and \({\|\mathscr L_f\|}_{\mathbb B}={\big\||Df|\big\|}_{L^2(\mm)}\)
is satisfied for every \(f\in W^{1,2}(\X,\sfd,\mm)\),
as proven in \cite[Proposition 5.10]{DMGPS18}.
\fr}\end{remark}
The following definition -- which has been introduced in \cite{Gigli12}
-- plays a key role in this paper.
\begin{definition}[Infinitesimal Hilbertianity]
We say that a metric measure space \((\X,\sfd,\mm)\) is
\emph{infinitesimally Hilbertian} provided \(W^{1,2}(\X,\sfd,\mm)\)
is a Hilbert space.
\end{definition}
The following result provides a sufficient condition for the infinitesimal
Hilbertianity to hold.
\begin{proposition}\label{prop:suff_cond_iH}
Let \((\X,\sfd,\mm)\) be a metric measure space. Suppose that
\begin{equation}\label{eq:pr_der}
|\b+\b'|^2+|\b-\b'|^2=2\,|\b|^2+2\,|\b'|^2\;\;\;\mm\text{-a.e.}
\quad\text{ for every }\b,\b'\in{\rm Der}^{2,2}(\X;\mm).
\end{equation}
Then \((\X,\sfd,\mm)\) is infinitesimally Hilbertian.
\end{proposition}
\begin{proof}
By integrating \eqref{eq:pr_der} we see that the norm
\(\|\cdot\|_2\) on \({\rm Der}^{2,2}(\X;\mm)\) (defined
in Remark \ref{rmk:mathscr_L_f}) satisfies the parallelogram rule,
whence the dual space \(\mathbb B\) of
\(\big({\rm Der}^{2,2}(\X;\mm),\|\cdot\|_2\big)\) is a Hilbert space.
Therefore, we know from Remark \ref{rmk:mathscr_L_f} that for every
\(f,g\in W^{1,2}(\X,\sfd,\mm)\) it holds that
\[\begin{split}
{\big\||D(f+g)|\big\|}^2_{L^2(\mm)}+{\big\||D(f-g)|\big\|}^2_{L^2(\mm)}
&={\|\mathscr L_{f+g}\|}^2_{\mathbb B}+{\|\mathscr L_{f-g}\|}^2_{\mathbb B}
={\|\mathscr L_f+\mathscr L_g\|}^2_{\mathbb B}+
{\|\mathscr L_f-\mathscr L_g\|}^2_{\mathbb B}\\
&=2\,{\|\mathscr L_f\|}^2_{\mathbb B}+2\,{\|\mathscr L_g\|}^2_{\mathbb B}
=2\,{\big\||Df|\big\|}^2_{L^2(\mm)}+2\,{\big\||Dg|\big\|}^2_{L^2(\mm)},
\end{split}\]
which proves that \(W^{1,2}(\X,\sfd,\mm)\) is a Hilbert space, as required.
\end{proof}

Finally, we conclude the subsection by reporting the following consequence
of the metric version of Smirnov's superposition principle, which has been
proven by E.\ Paolini and E.\ Stepanov in \cite{PaoSte12}.
\begin{theorem}[Superposition principle]\label{thm:superposition_principle}
Let \((\X,\sfd,\mm)\) be a metric measure space with \(\mm\) finite.
Let \(\b\in{\rm Der}^{2,2}(\X;\mm)\). Then there exists a finite,
non-negative Borel measure \(\ppi\) on \(C\big([0,1],\X\big)\), concentrated
on the set of non-constant Lipschitz curves on \(\X\) having constant
speed, such that
\begin{subequations}\begin{align}
\int g\,\b(f)\,\d\mm&=
\int\!\!\!\int_0^1 g(\gamma_t)\,(f\circ\gamma)'_t\,\d t\,\d\ppi(\gamma),
\label{eq:superposition_principle_1}\\
\int g\,|\b|\,\d\mm&=
\int\!\!\!\int_0^1 g(\gamma_t)\,|\dot\gamma_t|\,\d t\,\d\ppi(\gamma)
\label{eq:superposition_principle_2}
\end{align}\end{subequations}
for every \(f,g\in\LIP_{bs}(\X)\).
\end{theorem}
\begin{proof}
Combine \cite[Theorem 4.9]{DMGPS18} with \cite[Lemma 6.1]{DMGPS18}.
\end{proof}
\section{Monotone approximation of generalised metrics}\label{s:approx_metr}
\subsection{Set-up and auxiliary results}
We begin with some classical definitions.
A norm \(\n\) defined on a finite-dimensional vector space \(V\)
is said to be \emph{smooth} provided it is of class \(C^\infty\)
on \(V\setminus\{0\}\). In addition, we say that \(\n\)
is \emph{strongly convex} if the Hessian matrix of \(\n^2\) at any
vector \(v\in V\setminus\{0\}\) is positive definite.
With the notation \(W\leq V\) we intend that \(W\) is a vector subspace of \(V\).
\medskip

By \emph{smooth manifold} we shall always mean a connected differentiable
manifold of class \(C^\infty\).
Given a smooth manifold \(\M\) and a smooth vector bundle \((\E,\pi)\) over \(\M\),
we say that a function \(F\colon\E\to[0,+\infty)\) is a \emph{Finsler metric}
over \(\E\) if it is continuous, it is smooth on the complement of the zero section,
and \(F|_{\E_x}\) is a strongly convex norm on the fiber
\(\E_x\coloneqq\pi^{-1}(x)\) for every \(x\in\M\).

By Finsler metric on \(\M\) we mean a Finsler metric \(F\) over the tangent
bundle \(\T\M\). In this case, we also say that the couple \((\M,F)\) is a
\emph{Finsler manifold}.
(In the literature, \((\M,F)\) is often referred to as a
reversible Finsler manifold; cf., for instance, the monograph
\cite{BCS00}.)
\begin{definition}[Generalised norm]
Let \(V\) be a vector space. Then a function \(\n\colon V\to[0,+\infty]\)
is said to be a \emph{generalised norm} if there exists a vector
subspace \(D(\n)\neq\{0\}\) of \(V\) such that \(\n|_{D(\n)}\) is a norm
on \(D(\n)\) and \(\n(v)=+\infty\) holds for every \(v\in V\setminus D(\n)\).
\end{definition}
\begin{theorem}[Definition of continuous distribution]\label{thm:rank_vary_distr}
Let \(\M\) be a smooth manifold and let \((\E,\pi)\) be a smooth vector bundle
over \(\M\). Let \(\{V_x\}_{x\in\M}\) be a family of vector spaces
such that \(V_x\leq\E_x\) for all \(x\in\M\).
Then the following conditions are equivalent:
\begin{itemize}
\item[\(\rm i)\)] Given \(\bar x\in\M\) and \(\bar v\in V_{\bar x}\),
there exists a continuous section \(v\) of \(\E\), defined on some
neighbourhood \(U\) of \(\bar x\), such that \(v(\bar x)=\bar v\)
and \(v(x)\in V_x\) for every \(x\in U\).
\item[\(\rm ii)\)] Given \(\bar x\in\M\), there exist finitely
many continuous sections \(v_1,\ldots,v_k\) of \(\E\), defined on some
neighbourhood \(U\) of \(\bar x\), such that
\(V_x={\rm span}\big\{v_1(x),\ldots,v_k(x)\big\}\)
for every \(x\in U\).
\item[\(\rm iii)\)] Given \(\bar x\in\M\), there exist a neighbourhood
\(U\) of \(\bar x\), a smooth vector bundle \(\tilde\E\) over \(U\), and
a continuous vector bundle morphism \(\psi\colon\tilde\E\to\E|_U\), such
that \(V_x=\psi(\tilde\E_x)\) for all \(x\in U\).
\end{itemize}
If the above conditions are satisfied, we say that \(\{V_x\}_{x\in\M}\)
is a \emph{continuous distribution (of possibly varying rank)} over \(\M\).
Moreover, we can assume that \(k\) and the rank of \(\tilde\E\) are at most
\(d\,2^{2\cdot 5^n-1}\), where \(d\) is the rank of \(\E\) and
\(n\) is the dimension of \(\M\).
\end{theorem}
\begin{proof}
The real novelty of the theorem is the implication \({\rm i)}\Longrightarrow{\rm ii)}\).\\
{\color{blue}\({\rm i)}\Longrightarrow{\rm ii)}\)} Suppose item i) holds.
Given a point \(\bar x\in\M\), we can choose an open set \(U'\subseteq\R^n\)
containing \(0\) and a map \(\varphi\colon U'\to\M\) satisfying
\(\varphi(0)=\bar x\) that is a homeomorphism with its image.
Possibly shrinking \(U'\), we can assume there exists a Finsler metric
\(F\) over \(\E|_U\), where \(U\coloneqq\varphi(U')\). Fix any
radius \(\lambda>0\) such that \(\bar B_\lambda(0)\subseteq U'\) and
call \(K\coloneqq\varphi\big(\bar B_\lambda(0)\big)\subseteq U\).
For any \(i=1,\ldots,d\) we set \(C_i\coloneqq\{x\in K\,:\,\dim V_x\leq i\}\).
In order to prove ii), it would be enough to find some finite families
\(\mathcal F_1\subseteq\ldots\subseteq\mathcal F_d\) of continuous
sections of \(\E|_K\) such that for any \(i=1,\ldots,d\) it holds that
\[\begin{split}
V_x=\mathcal F_i(x)\coloneqq{\rm span}\big\{v(x)\;\big|\;v\in\mathcal F_i\big\}&
\quad\text{ for every }x\in C_i,\\
\mathcal F_i(x)\leq V_x\;\text{ and }\,\dim\mathcal F_i(x)\geq i&
\quad\text{ for every }x\in K\setminus C_i.
\end{split}\]
We build \(\mathcal F_1,\ldots,\mathcal F_d\) via a recursive argument.
Suppose to have already defined \(\mathcal F_1,\ldots,\mathcal F_i\) for
some \(i<d\). Notice that item i) grants that the function
\(\M\ni x\mapsto\dim V_x\) is lower semicontinuous, thus \(C_i\) is a
compact set. For any \(j\in\N\) we define the compact set \(K_j\subseteq K\) as
\[
K_j\coloneqq\varphi\bigg(\bigg\{y\in\bar B_\lambda(0)\;\bigg|\;
\frac{\lambda}{2^{j+1}}\leq{\rm dist}\big(\varphi^{-1}(C_i),y\big)
\leq\frac{\lambda}{2^{j-1}}\bigg\}\bigg).
\]
Observe that \(\bigcup_{j\in\N}K_j=K\setminus C_i\) and that
\(\mathring{K_j}\cap\mathring{K_{j'}}=\emptyset\) for all
\(j,j'\in\N\) such that \(|j-j'|\) is even.

Let \(j\in\N\) be fixed. For any \(x\in K_j\), we choose a vector
\(\bar w_x\in V_x\) such that \(F(\bar w_x)=1\) and
\(\dim\big(\mathcal F_i(x)+\R\,\bar w_x\big)\geq i+1\).
%, which satisfies \(\bar w_x\notin\mathcal F_i(x)\) whenever \(\mathcal F_i(x)\neq V_x\).
By item i), we can find a neighbourhood \(W_x\subseteq U\) of \(x\)
and a continuous section \(w_x\) of \(\E|_{W_x}\), such that \(w_x(x)=\bar w_x\)
and \(w_x(z)\in V_z\) for all \(z\in W_x\). Possibly shrinking \(W_x\),
we can further assume that \(0<F\big(w_x(z)\big)\leq 2\) and
\(\dim\big(\mathcal F_i(z)+\R\,w_x(z)\big)\geq i+1\) hold for every point \(z\in W_x\).
By compactness of \(K_j\), we can thus find an open covering
\(W_1,\ldots,W_m\subseteq U\) of \(K_j\) and continuous sections
\(w_1,\ldots,w_m\) of \(\E|_{W_1},\ldots,\E|_{W_m}\), respectively,
such that \(0<F\big(w_\iota(x)\big)\leq 2\) and
\(\dim\big(\mathcal F_i(x)+\R\,w_\iota(x)\big)\geq i+1\)
for every \(\iota=1,\ldots,m\) and \(x\in W_\iota\).
By Lebesgue's number lemma, there exists \(r>0\) such that
any ball in \(\R^n\) of radius \(r\) centered at \(\varphi^{-1}(K_j)\)
is entirely contained in one of the sets \(\varphi^{-1}(W_1),\ldots,\varphi^{-1}(W_m)\).
Choose a maximal \(r\)-separated subset \(S\) of \(\varphi^{-1}(K_j)\),
i.e., \(S\) is maximal among all subsets satisfying \(|p-q|\geq r\)
for every \(p,q\in S\) with \(p\neq q\). Note that \(S\) is a finite
set by compactness of \(\varphi^{-1}(K_j)\). For any \(p\in S\),
call \(G_p\coloneqq\varphi\big(B_r(p)\big)\) and pick \(\iota(p)\in\{1,\ldots,m\}\)
such that \(G_p\subseteq W_{\iota(p)}\). By definition of \(S\), it holds
that \(K_j\subseteq\bigcup_{p\in S}G_p\). Moreover, given any \(p\in S\)
we have that the balls \(\big\{B_{r/2}(q)\,:\,q\in S\setminus\{p\},\,|p-q|<2r\big\}\)
are pairwise disjoint and contained in \(B_{5r/2}(p)\setminus B_{r/2}(p)\),
whence accordingly \(\#\big\{q\in S\setminus\{p\}\;:\;|p-q|<2r\big\}\leq 5^n-1\)
for all \(p\in S\). Therefore, we can take a partition \(S=S_1\cup\ldots\cup S_{5^n}\)
with the property that \(G_p\cap G_q=\emptyset\) whenever \(\ell=1,\ldots,5^n\)
and \(p,q\in S_\ell\) satisfy \(p\neq q\).
Given \(\ell=1,\ldots,5^n\) and \(p\in S_\ell\),
we can pick a continuous function \(\psi_p\colon K\to[0,1]\) satisfying \(\psi_p=0\)
on \(K_j\setminus G_p\) and \(\psi_p>0\) on \(K_j\cap G_p\). For any multi-index
\(\alpha=(\alpha_2,\ldots,\alpha_{5^n})\in\{-1,1\}^{5^n-1}\), we define
\[
v_{j\alpha}(x)\coloneqq\sum_{p\in S_1}\psi_p(x)\,w_{\iota(p)}(x)+
\sum_{\ell=2}^{5^n}\alpha_\ell\sum_{p\in S_\ell}\psi_p(x)\,w_{\iota(p)}(x)
\quad\text{ for every }x\in K_j.
\]
Notice that \(F\big(v_{j\alpha}(x)\big)\leq 2\cdot 5^n\) for every
\(\alpha\in\{-1,1\}^{5^n-1}\) and \(x\in K_j\).
For any \(j\in\N\) we fix a continuous functions \(\eta_j\colon K\to[0,1]\)
such that \(\eta_j=0\) on \(K\setminus\mathring{K_j}\) and \(\eta_j>0\) on
\(\mathring{K_j}\). Then we set
\[
\mathcal F'_{i+1}\coloneqq\bigg\{\sum_{j\text{ even}}\frac{\eta_j}{2^j}\,v_{j\alpha}
\pm\sum_{j\text{ odd}}\frac{\eta_j}{2^j}\,v_{j\beta}\;\bigg|\;
\alpha,\beta\in\{-1,1\}^{5^n-1}\bigg\}.
\]
Therefore, it follows from the construction that the family
\(\mathcal F_{i+1}\coloneqq\mathcal F_i\cup\mathcal F'_{i+1}\) of continuous
sections of \(\E|_K\) satisfies \(\dim\mathcal F_{i+1}(x)\geq i+1\) for
every \(x\in K\), as required. Observe also that \(\#\mathcal F'_i\leq 2^{2\cdot 5^n-1}\)
for all \(i=1,\ldots,d\), thus the cardinality of \(\mathcal F\coloneqq\mathcal F_d\)
does not exceed \(d\,2^{2\cdot 5^n-1}\). This proves ii).\\
{\color{blue}\({\rm ii)}\Longrightarrow{\rm iii)}\)} Suppose item ii) holds.
Given a point \(\bar x\in\M\), pick a neighbourhood \(U\) of \(\bar x\)
and some continuous sections \(v_1,\ldots,v_k\) of \(\E|_U\) such
that \(V_x={\rm span}\big\{v_1(x),\ldots,v_k(x)\big\}\) for every \(x\in U\).
Let us define \(\tilde\E\coloneqq U\times\R^k\) and the continuous vector bundle
morphism \(\psi\colon\tilde E\to\E|_U\) as
\[
\psi(x,\lambda)\coloneqq\sum_{i=1}^k\lambda_i\,v_i(x)
\quad\text{ for every }x\in U\text{ and }\lambda=(\lambda_1,\ldots,\lambda_k)\in\R^k.
\]
Therefore, we conclude that
\(\psi(\tilde\E_x)={\rm span}\big\{v_1(x),\ldots,v_k(x)\big\}=V_x\)
for all \(x\in U\), thus proving iii).\\
{\color{blue}\({\rm iii)}\Longrightarrow{\rm i)}\)} Suppose item iii) holds.
Fix \(\bar x\in\M\) and \(\bar v\in V_{\bar x}\). There exist a smooth
vector bundle \(\tilde\E\) over some neighbourhood \(U'\) of \(\bar x\) and a continuous
vector bundle morphism \(\psi\colon\tilde\E\to\E|_{U'}\) such that \(V_x=\psi(\tilde\E_x)\)
for all \(x\in U'\). Choose any \(\bar w\in\tilde\E_{\bar x}\) for which
\(\psi(\bar w)=\bar v\). Then we can find a neighbourhood \(U\subseteq U'\) of
\(\bar x\) and a continuous section \(w\) of \(\tilde\E|_U\) such that
\(w(\bar x)=\bar w\). Now let us define \(v(x)\coloneqq\psi\big(w(x)\big)\)
for every \(x\in U\). Therefore, it holds that \(v\) is a continuous section of
\(\E|_U\) such that \(v(\bar x)=\bar v\) and \(v(x)\in V_x\) for all \(x\in U\),
thus proving i).
\end{proof}
\begin{remark}\label{rmk:dim_lsc}{\rm
As already observed during the proof of Theorem \ref{thm:rank_vary_distr},
the function \(\M\ni x\mapsto\dim V_x\) is lower semicontinuous
whenever \(\{V_x\}_{x\in\M}\) is a continuous distribution over \(\M\).
\fr}\end{remark}
\begin{definition}[Generalised metric]\label{def:gen_metric}
Let \(\M\) be a smooth manifold, \((\E,\pi)\) a smooth vector
bundle over \(\M\). Then a \emph{generalised metric} over \(\E\) is a lower
semicontinuous function \(\rho\colon\E\to[0,+\infty]\) having the
following properties:
\begin{itemize}
\item[\(\rm i)\)] \(\rho_x\coloneqq\rho|_{\E_x}\) is a generalised
norm on \(\E_x\) for every \(x\in\M\).
\item[\(\rm ii)\)] The family \(\big\{D(\rho_x)\}_{x\in\M}\)
is a continuous distribution.
\end{itemize}
In the case \(\E=\T\M\), we just say that \(\rho\) is a generalised metric on \(\M\).
\end{definition}
Let us fix some notation: given any vector subspace \(V\leq\R^d\), we denote by
\(V^\perp\) its orthogonal complement (with respect to the Euclidean norm).
We denote by \(\mathbb S^{d-1}\) the Euclidean unit sphere in \(\R^d\), i.e., the
set of all points \(x=(x_1,\ldots,x_d)\in\R^d\) such that \(x_1^2+\ldots+x_d^2=1\),
while \(x\cdot y\) stands for the Euclidean scalar product between \(x\in\R^d\)
and \(y\in\R^d\). Finally, given a metric space \((\X,\sfd)\) and two compact
non-empty sets \(A,B\subseteq\X\), we shall denote by \(\sfd_H(A,B)\) the Hausdorff
distance between \(A\) and \(B\), i.e.,
\begin{equation}\label{eq:def_Hausd_dist}
\sfd_H(A,B)\coloneqq\max\bigg\{\sup_{x\in A}\inf_{y\in B}\sfd(x,y),\,
\sup_{y\in B}\inf_{x\in A}\sfd(x,y)\bigg\}.
\end{equation}
\begin{lemma}\label{lem:approx_norm}
Let \(\{0\}\neq V\leq\R^d\) be given. Let \(\|\cdot\|\) be a norm on \(V\).
Fix a constant \(\lambda>0\) and a norm \(\|\cdot\|'\) on \(\R^d\) such that
\(\|v\|'<\|v\|\) for all \(v\in V\setminus\{0\}\). Then there exists a norm
\(\n\) on \(\R^d\) such that the following properties are satisfied:
\[\begin{split}
\n(v)=\|v\|&\quad\text{ for every }v\in V,\\
\n(v)>\|v\|'&\quad\text{ for every }v\in\R^d\setminus\{0\},\\
\n(v)\geq\lambda&\quad\text{ for every }v\in V^\perp\cap\mathbb S^{d-1}.
\end{split}\]
\end{lemma}
\begin{proof}
Call \(\lambda'\coloneqq\lambda+\max\big\{\|v\|'\,:\,v\in\mathbb S^{d-1}\big\}\)
and \(k\coloneqq\dim V\). Fix any orthonormal basis \(e_1,\ldots,e_d\) of \(\R^d\)
(equipped with the Euclidean norm) such that \(e_1,\ldots,e_k\) is a basis of \(V\).
Hence, we define the norm \(\n\) on \(\R^d\) as follows:
given any \(\alpha=(\alpha_1,\ldots,\alpha_d)\in\R^d\), we set
\begin{equation}\label{eq:def_n}
\n(\alpha_1\,e_1+\ldots+\alpha_d\,e_d)\coloneqq
\|\alpha_1\,e_1+\ldots+\alpha_k\,e_k\|+\lambda'\,\big|(\alpha_{k+1},\ldots,\alpha_d)\big|.
\end{equation}
It directly follows from its very definition that the norm \(\n\) satisfies
\(\n(v)=\|v\|\) for every \(v\in V\) and \(\n(v)=\lambda'>\lambda\)
for every \(v\in V^\perp\cap\mathbb S^{d-1}\). Finally, for any
choice of \(\alpha=(\alpha_1,\ldots,\alpha_d)\in\R^d\) with
\((\alpha_{k+1},\ldots,\alpha_d)\neq 0\) we have that
\[\begin{split}
\n(\alpha_1\,e_1+\ldots+\alpha_d\,e_d)&=\|\alpha_1\,e_1+\ldots+\alpha_k\,e_k\|
+\lambda'\,\big|(\alpha_{k+1},\ldots,\alpha_d)\big|\\
&\geq\|\alpha_1\,e_1+\ldots+\alpha_k\,e_k\|'
+\lambda'\,\big|(\alpha_{k+1},\ldots,\alpha_d)\big|\\
&>\|\alpha_1\,e_1+\ldots+\alpha_k\,e_k\|'+\bigg\|\frac{\alpha_{k+1}\,e_{k+1}+\ldots
+\alpha_d\,e_d}{\big|(\alpha_{k+1},\ldots,\alpha_d)\big|}\bigg\|'
\,\big|(\alpha_{k+1},\ldots,\alpha_d)\big|\\
&\geq\|\alpha_1\,e_1+\ldots+\alpha_d\,e_d\|',
\end{split}\]
thus completing the proof of the statement.
\end{proof}
In the following results, we shall consider the trivial bundle \(\M\times\R^d\)
over \(\M\). Given any \(x\in\M\), a vector subspace of the fiber of
\(\M\times\R^d\) at \(x\) is of the form \(\{x\}\times V\), for some vector
subspace \(V\leq\R^d\). For simplicity, we will always implicitly identify
\(\{x\}\times V\) with the vector space \(V\) itself.
\begin{lemma}\label{lem:distrib_cont}
Let \(\M\) be a smooth manifold and \(\rho\) a generalised metric
over \(\M\times\R^d\). Fix \(\bar x\in\M\) and any norm \(\|\cdot\|\)
on \(\R^d\). Call \(V_x\coloneqq D(\rho_x)\) for every \(x\in\M\)
and \(k\coloneqq\dim V_{\bar x}\). Then for any \(\eps>0\) there
exists a neighbourhood \(U\) of \(\bar x\) such that
\[
\sfd_H(V_{\bar x}\cap\mathbb S^{d-1},V_x\cap\mathbb S^{d-1})\leq\eps
\quad\text{ for every }x\in U\text{ with }\dim V_x=k,
\]
where the Hausdorff distance \(\sfd_H\) is computed with respect to the
norm \(\|\cdot\|\).
\end{lemma}
\begin{proof}
Since \(\{V_x\}_{x\in\M}\) is a continuous distribution,
we can find a neighbourhood \(U'\) of \(\bar x\) and some continuous
maps \(v_1,\ldots,v_{k'}\colon U'\to\R^d\) such that
\(V_x={\rm span}\big\{v_1(x),\ldots,v_{k'}(x)\big\}\) for all \(x\in U'\).
Up to relabelling, we can assume that \(v_1(\bar x),\ldots,v_k(\bar x)\)
constitute a basis of \(V_{\bar x}\). Then there is a neighbourhood
\(U\subseteq U'\) of \(\bar x\) such that \(v_1(x),\ldots,v_k(x)\)
are linearly independent for all \(x\in U\).
Define \(W_x\coloneqq{\rm span}\big\{v_1(x),\ldots,v_k(x)\big\}\)
for every point \(x\in U\). Let us apply a Gram--Schmidt orthogonalisation
process to the vector fields \(v_1,\ldots,v_k\), with respect to the Euclidean
norm \(|\cdot|\):
\[\begin{split}
w_1(x)&\coloneqq\frac{v_1(x)}{\big|v_1(x)\big|},\\
w_2(x)&\coloneqq\frac{v_2(x)-\big(v_2(x)\cdot w_1(x)\big)w_1(x)}
{\big|v_2(x)-\big(v_2(x)\cdot w_1(x)\big)w_1(x)\big|},\\
&\vdots\\
w_k(x)&\coloneqq\frac{v_k(x)-\sum_{i=1}^{k-1}\big(v_k(x)\cdot w_i(x)\big)w_i(x)}
{\big|v_k(x)-\sum_{i=1}^{k-1}\big(v_k(x)\cdot w_i(x)\big)w_i(x)\big|}
\end{split}\]
for every \(x\in U\). Therefore, the resulting continuous maps
\(w_1,\ldots,w_k\colon U\to\R^d\) satisfy
\[\begin{split}
w_i(x)\cdot w_j(x)=\delta_{ij}&\quad\text{ for every }i,j=1,\ldots,k\text{ and }x\in U,\\
W_x={\rm span}\big\{w_1(x),\ldots,w_k(x)\big\}&\quad\text{ for every }x\in U.
\end{split}\]
Fix any \(C>0\) such that \(\|v\|\leq C\,|v|\) for all \(v\in\R^d\).
Possibly shrinking \(U\), we can assume that
\begin{equation}\label{eq:estim_w_i}
\big|w_i(\bar x)-w_i(x)\big|\leq\frac{\eps}{C\sqrt k}
\quad\text{ for every }i=1,\ldots,k\text{ and }x\in U.
\end{equation}
Since \(\big\{\sum_{i=1}^k q_i\,w_i(x)\,:
\,q=(q_1,\ldots,q_k)\in\mathbb Q^k\cap\mathbb S^{k-1}\big\}\)
is dense in \(W_x\cap\mathbb S^{d-1}\) for all \(x\in U\), we have
\[\begin{split}
\sfd_H(W_{\bar x}\cap\mathbb S^{d-1},W_x\cap\mathbb S^{d-1})
&\overset{\eqref{eq:def_Hausd_dist}}\leq
\sup_{q\in\mathbb Q^k\cap\mathbb S^{k-1}}
\bigg\|\sum_{i=1}^k q_i\big(w_i(\bar x)-w_i(x)\big)\bigg\|\\
&\overset{\phantom{\eqref{eq:def_Hausd_dist}}}\leq
C\sup_{q\in\mathbb Q^k\cap\mathbb S^{k-1}}
\bigg(\sum_{i=1}^k q_i^2\bigg)^{\nicefrac{1}{2}}
\bigg(\sum_{i=1}^k\big|w_i(\bar x)-w_i(x)\big|^2\bigg)^{\nicefrac{1}{2}}
\overset{\eqref{eq:estim_w_i}}\leq\eps
\end{split}\]
for every \(x\in U\). The statement follows by noticing
that \(W_x=V_x\) if \(x\in U\) satisfies \(\dim V_x=k\).
\end{proof}
\begin{definition}[Continuous metric]\label{def:cont_metric}
Let \(\M\) be a smooth manifold. Let \((\E,\pi)\) be a smooth vector bundle
over \(\M\). Then a \emph{continuous metric} \(\sigma\) over \(\E\) is a
continuous function \(\sigma\colon\E\to[0,+\infty)\) such that \(\sigma|_{\E_x}\)
is a norm on the fiber \(\E_x\) for every point \(x\in\M\).
\end{definition}
\begin{lemma}\label{lem:main_lemma}
Let \(\M\) be a smooth manifold, \(\rho\) a generalised metric
over \(\M\times\R^d\). Call \(V_x\coloneqq D(\rho_x)\) for every \(x\in\M\).
Fix \(\bar x\in\M\) and two constants \(\eps,\lambda>0\).
Let \(\sigma\) be a continuous metric over \(\M\times\R^d\) such that
\(\sigma(x,v)<\rho(x,v)\) for every \(x\in\M\)
and \(v\in\R^d\setminus\{0\}\). Then there exist
a smooth, strongly convex norm \(\n\) on \(\R^d\) and a neighbourhood \(U\)
of \(\bar x\) such that the following properties hold:
\begin{itemize}
\item[\(\rm i)\)] \(\big|\n(v)-\rho(\bar x,v)\big|\leq\eps\) for
every \(v\in V_{\bar x}\cap\mathbb S^{d-1}\).
\item[\(\rm ii)\)] \(\sigma(x,v)<\n(v)<\rho(x,v)\) for every \(x\in U\)
and \(v\in\R^d\setminus\{0\}\).
\item[\(\rm iii)\)] \(\dim V_{\bar x}=\min\big\{\dim V_x\,:\,x\in U\big\}\) and
\[
\n(v)\geq\lambda\quad\text{ for every }
x\in D\text{ and }v\in V_x^\perp\cap\mathbb S^{d-1},
\]
where we set \(D\coloneqq\big\{x\in U\,:\,\dim V_x=\dim V_{\bar x}\big\}\).
\end{itemize} 
\end{lemma}
\begin{proof}
We divide the proof into several steps:\\
{\color{blue}\textsc{Step 1.}} Lemma \ref{lem:approx_norm} grants the existence
of a norm \(\n'\) on \(\R^d\) such that
\begin{equation}\label{eq:main_lemma_1}\begin{split}
\n'(v)=\rho(\bar x,v)&\quad\text{ for every }v\in V_{\bar x},\\
\n'(v)>\sigma(\bar x,v)&\quad\text{ for every }v\in\R^d\setminus\{0\},\\
\n'(v)\geq\lambda+1&\quad\text{ for every }v\in V_{\bar x}^\perp\cap\mathbb S^{d-1}.
\end{split}\end{equation}
Given that \(\M\ni x\mapsto\dim V_x\) is lower semicontinuous, we can find a
neighbourhood \(U'\) of \(\bar x\) such that \(\dim V_{\bar x}\) is the minimum
of the function \(U'\ni x\mapsto \dim V_x\).\\
{\color{blue}\textsc{Step 2.}} The function
\(\mathbb S^{d-1}\ni v\mapsto\n'(v)-\sigma(\bar x,v)>0\)
is continuous by construction, thus there exists \(\eps'\in(0,\eps)\) such
that \(\sigma(\bar x,v)<\n'(v)-\eps'\) for all \(v\in\mathbb S^{d-1}\).
Choose any constant \(\delta>0\) such that
\((\lambda+1)(1-\delta)>\lambda\) and
\(\delta<\eps'/\max\big\{\n'(v)\,:\,v\in\mathbb S^{d-1}\big\}\).
Pick also \(\eps''\in(0,\eps')\) such that
\(\eps''<\delta\,\min\big\{\n'(v)\,:\,v\in\mathbb S^{d-1}\big\}\). Then it holds that
\begin{equation}\label{eq:main_lemma_2}
\sigma(\bar x,v)<\n'(v)-\eps'<\n'(v)-\eps''<\rho(\bar x,v)
\quad\text{ for every }v\in\mathbb S^{d-1}.
\end{equation}
Being \((x,v)\mapsto\rho(x,v)-\n'(v)+\eps''\) lower semicontinuous
and \((x,v)\mapsto\n'(v)-\eps'-\sigma(x,v)\) continuous, we deduce from
\eqref{eq:main_lemma_2} that there exists a neighbourhood \(U''\subseteq U'\)
of \(\bar x\) such that
\begin{equation}\label{eq:main_lemma_3}
\sigma(x,v)<\n'(v)-\eps'<\n'(v)-\eps''<\rho(x,v)
\quad\text{ for every }x\in U''\text{ and }v\in\mathbb S^{d-1}.
\end{equation}
Let us define \(\n''\coloneqq(1-\delta)\,\n'\). Our choice of \(\delta\)
and \(\eps''\) yields \(\n'(v)-\eps'<(1-\delta)\,\n'(v)<\n'(v)-\eps''\)
for every \(v\in\mathbb S^{d-1}\), which together with \eqref{eq:main_lemma_3}
imply that
\begin{equation}\label{eq:main_lemma_4}
\sigma(x,v)<\n''(v)<\rho(x,v)\quad\text{ for every }x\in U''
\text{ and }v\in\mathbb S^{d-1}.
\end{equation}
Moreover, for any \(v\in V_{\bar x}\cap\mathbb S^{d-1}\) it holds
\(\n''(v)>\n'(v)-\eps'>\rho(\bar x,v)-\eps\) by the first line
of \eqref{eq:main_lemma_1}.\\
{\color{blue}\textsc{Step 3.}} In light of \eqref{eq:main_lemma_4},
there exists a constant \(\delta'>0\) with \((\lambda+1)(1-\delta)-\delta'>\lambda\)
such that
\begin{equation}\label{eq:main_lemma_5}\begin{split}
\sigma(x,v)<\n''(v)-\delta'<\n''(v)+\delta'<\rho(x,v)&
\quad\text{ for every }x\in U''\text{ and }v\in\mathbb S^{d-1},\\
\n''(v)-\delta'>\rho(\bar x,v)-\eps&
\quad\text{ for every }v\in V_{\bar x}\cap\mathbb S^{d-1}.
\end{split}\end{equation}
Choose any smooth norm \(\|\cdot\|\) on \(\R^d\) such that
\(\big|\|v\|-\n''(v)\big|\leq\delta'/2\) holds for all \(v\in\mathbb S^{d-1}\),
whose existence follows, e.g., from
\cite[Theorem 103]{HM14}.
Then let us finally define the sought norm \(\n\) as
\(\n(v)\coloneqq\|v\|+\delta'|v|/2\) for every \(v\in\R^d\).
Clearly, it is a smooth and strongly convex norm by construction.
Moreover, it can be immediately checked that \(\n\) satisfies
\begin{equation}\label{eq:main_lemma_6}
\big|\n(v)-\n''(v)\big|\leq\delta'\quad\text{ for every }v\in\mathbb S^{d-1}.
\end{equation}
Accordingly, by combining \eqref{eq:main_lemma_5} with \eqref{eq:main_lemma_6}
we obtain that
\begin{equation}\label{eq:main_lemma_7}\begin{split}
\sigma(x,v)<\n(v)<\rho(x,v)&
\quad\text{ for every }x\in U''\text{ and }v\in\mathbb S^{d-1},\\
\n(v)>\rho(\bar x,v)-\eps&\quad\text{ for every }v\in V_{\bar x}\cap\mathbb S^{d-1}.
\end{split}\end{equation}
{\color{blue}\textsc{Step 4.}} Observe that \eqref{eq:main_lemma_6} and the
third line of \eqref{eq:main_lemma_1} give
\begin{equation}\label{eq:main_lemma_8}
\n(v)\geq(1-\delta)\,\n'(v)-\delta'\geq(\lambda+1)(1-\delta)-\delta'
\quad\text{ for every }v\in V_{\bar x}^\perp\cap\mathbb S^{d-1}.
\end{equation}
Since \(\M\ni x\mapsto V_x\) is a continuous distribution and
\((\lambda+1)(1-\delta)-\delta'>\lambda\), we deduce from \eqref{eq:main_lemma_8}
that for some neighbourhood \(U\subseteq U''\) of \(\bar x\) we have
\[
\n(v)\geq\lambda\quad\text{ for every }x\in U\text{ with }\dim V_x=\dim V_{\bar x}
\text{ and }v\in V_x^\perp\cap\mathbb S^{d-1}.
\]
Therefore, item iii) is verified (recall the last claim in \textsc{Step 1}).
Finally, we deduce from \eqref{eq:main_lemma_7} that also items i) and ii) hold,
thus concluding the proof of the statement.
\end{proof}
\subsection{The approximation result}
Let \(\M\) be a smooth manifold and let \(\rho\) be a generalised metric
over \(\M\times\R^d\). Calling \(V_x\coloneqq D(\rho_x)\) for every \(x\in\M\),
it holds that \(\M\ni x\mapsto\dim V_x\) is a lower semicontinuous function
(recall Remark \ref{rmk:dim_lsc}), thus for any \(x\in\M\) there exists \(r_x>0\)
such that
\[
\dim V_x=\min\big\{\dim V_y\;\big|\;y\in B_{r_x}(x)\big\}.
\]
Let us define
\begin{equation}\label{eq:def_G_n}
G_n\coloneqq\big\{x\in\M\;\big|\;r_x\geq 1/n\big\}\quad\text{ for every }n\in\N.
\end{equation}
Observe that for any point \(x\in\M\) there exists \(\bar n\in\N\) such that
\(x\in\bigcap_{n\geq\bar n}G_n\).
\begin{proposition}\label{prop:constr_g_n}
Let \(\M\) be a smooth manifold. Let \(\sfd\) be any distance on \(\M\)
that induces the manifold topology. Fix a generalised metric
\(\rho\) over \(\M\times\R^d\). Then there exists a sequence \((F_n)_n\) of Finsler
metrics over \(\M\times\R^d\) such that the following properties are satisfied:
\begin{itemize}
\item[\(\rm a)\)] \(F_{n-1}(x,v)<F_n(x,v)<\rho(x,v)\) for every \(n\in\N\),
\(x\in\M\), and \(v\in\R^d\setminus\{0\}\).
\item[\(\rm b)\)] Given any \(n\in\N\), it holds that
\[
F_n(x,v)\geq n\quad\text{ for every }x\in G_n
\text{ and }v\in V_x^\perp\cap\mathbb S^{d-1},
\]
where \(V_x\coloneqq D(\rho_x)\) for all \(x\in\M\) and the set \(G_n\) 
is defined as in \eqref{eq:def_G_n}.
\item[\(\rm c)\)] For any \(n\in\N\) there exists a countable
set \(S_n\subseteq\M\) such that
\[
\big|F_n(z,v)-\rho(z,v)\big|\leq\frac{1}{n}\quad\text{ for every }
z\in S_n\text{ and }v\in V_z\cap\mathbb S^{d-1}.
\]
\item[\(\rm d)\)] Given any \(n\in\N\), \(x\in G_n\), and \(v\in\R^d\),
there exists a point \(z\in S_n\cap B_{1/n}(x)\) such that
\[
F_n(x,v)\geq F_n(z,v)\quad\text{ and }\quad
\sfd_H(V_z\cap\mathbb S^{d-1},V_x\cap\mathbb S^{d-1})<\frac{1}{n},
\]
where the Hausdorff distance \(\sfd_H\) is computed with respect to
the norm \(F_n(z,\cdot)+|\cdot|\).
\end{itemize}
\end{proposition}
\begin{proof}
We recursively define the Finsler metrics \(F_n\colon\M\times\R^d\to[0,+\infty)\).
Suppose to have already defined \(F_0,\ldots,F_{n-1}\) for some \(n\in\N\),
where \(F_0\coloneqq 0\). By using Lemma \ref{lem:distrib_cont}, Lemma 
\ref{lem:main_lemma}, and the paracompactness of \(\M\), we can find a family
\(\big\{(U^n_i,z^n_i,\n^n_i)\,:\,i\in\N\big\}\) such that:
\begin{itemize}
\item[\(\rm i)\)] \(\{U^n_i\}_{i\in\N}\) is a locally finite, open covering
of \(\M\), and \({\rm diam}(U^n_i)<1/n\) for every \(i\in\N\).
\item[\(\rm ii)\)] \(z^n_i\in U^n_i\) and
\(\dim V_{z^n_i}=\min\{\dim V_x\,:\,x\in U^n_i\}\) for every \(i\in\N\).
\item[\(\rm iii)\)] Given any \(i\in\N\), we have that \(\n^n_i\) is a smooth,
strongly convex norm on \(\R^d\) that satisfies
\[
F_{n-1}(x,v)<\n^n_i(v)<\rho(x,v)\quad\text{ for every }x\in U^n_i
\text{ and }v\in\R^d\setminus\{0\}.
\]
\item[\(\rm iv)\)] \(\big|\n^n_i(v)-\rho(z^n_i,v)\big|\leq 1/n\) for
every \(i\in\N\) and \(v\in V_{z^n_i}\cap\mathbb S^{d-1}\).
\item[\(\rm v)\)] Calling
\(D^n_i\coloneqq\big\{x\in U^n_i\,:\,\dim V_x=\dim V_{z^n_i}\big\}\)
for all \(i\in\N\), we have that \(\n^n_i(v)\geq n\) for every \(x\in D^n_i\)
and \(v\in V_x^\perp\cap\mathbb S^{d-1}\).
\item[\(\rm vi)\)] Given any \(i\in\N\), it holds that
\[
\sfd_H(V_{z^n_i}\cap\mathbb S^{d-1},V_x\cap\mathbb S^{d-1})<\frac{1}{n}
\quad\text{ for every }x\in D^n_i,
\]
where \(\sfd_H\) is intended with respect to \(\n^n_i+|\cdot|\).
\end{itemize}
Choose a partition of unity \(\{\varphi^n_i\,:\,i\in\N\}\subseteq C^\infty(\M)\)
subordinated to \(\{U^n_i\,:\,i\in\N\}\) such that \(\varphi^n_i(z^n_i)=1\)
for every \(i\in\N\). Let us define
\[
F_n(x,v)\coloneqq\sum_{i\in\N}\varphi^n_i(x)\,\n^n_i(v)
\quad\text{ for every }x\in\M\text{ and }v\in\R^d.
\]
Since each norm \(\n^n_i\) is smooth and strongly convex,
it can be readily checked that \(F_n\) is a Finsler metric over \(\M\times\R^d\).
Let us then conclude by verifying that \(F_n\) satisfies the desired properties:\\
{\color{blue}a)} It follows from iii) that \(F_{n-1}(x,v)<F_n(x,v)<\rho(x,v)\)
for all \(x\in\M\) and \(v\in\R^d\setminus\{0\}\).\\
{\color{blue}b)} Fix any point \(x\in G_n\). We claim that for any \(i\in\N\)
with \(x\in U^n_i\), it holds that \(x\in D^n_i\). Indeed, we know that \(r_x\geq 1/n\)
by definition of \(G_n\), whence it holds \(U^n_i\subseteq B_{r_x}(x)\) by i) and
accordingly \(\dim V_x=\dim V_{z^n_i}\) by ii). This shows that \(x\in D^n_i\),
thus proving the above claim.

Fix \(v\in V_x^\perp\cap\mathbb S^{d-1}\).
Therefore, we deduce from the previous claim and v) that
\[
\sum_{i\in\N}\varphi^n_i(x)=
\sum_{\substack{i\in\N: \\ x\in D^n_i}}\varphi^n_i(x)=1\quad\text{ and }\quad
F_n(x,v)=\sum_{\substack{i\in\N: \\ x\in D^n_i}}\varphi^n_i(x)\,\n^n_i(v)\geq n.
\]
{\color{blue}c)} Define \(S_n\coloneqq\{z^n_i\,:\,i\in\N\}\).
Notice that \(F_n(z^n_i,\cdot)=\n^n_i\) for any \(i\in\N\), whence
iv) gives c).\\
{\color{blue}d)} Fix \(n\in\N\), \(x\in G_n\), and \(v\in\R^d\).
Since the family \(\{i\in\N\,:\,x\in U^n_i\}\) is finite, we can
find \(j\in\N\) such that \(x\in U^n_j\) and 
\(F_n(z^n_j,v)=\min\big\{F_n(z^n_i,v)\,:\,i\in\N,\,x\in U^n_i\big\}\).
Consequently,
\[
F_n(x,v)=\sum_{\substack{i\in\N: \\ x\in U^n_i}}\varphi^n_i(x)\,\n^n_i(v)
=\sum_{\substack{i\in\N: \\ x\in U^n_i}}\varphi^n_i(x)\,F_n(z^n_i,v)
\geq F_n(z^n_j,v).
\]
Moreover, as in the proof of item b) we deduce that \(x\in D^n_j\).
Therefore, we know from item vi) that
\(\sfd_H(V_{z^n_j}\cap\mathbb S^{d-1},V_x\cap\mathbb S^{d-1})<1/n\),
where \(\sfd_H\) is taken with respect to \(\n^n_j+|\cdot|=F_n(z^n_j,\cdot)+|\cdot|\).
Notice also that \(z^n_j\in B_{1/n}(x)\), as \({\rm diam}(U^n_j)<1/n\) by i).
This gives the statement.
\end{proof}
\begin{lemma}\label{lem:metrics_conv}
Let \(\M\), \(\rho\), and \((F_n)_n\) be as in Proposition
\ref{prop:constr_g_n}. Then it holds that \(F_n(x,v)\nearrow\rho(x,v)\)
for every \(x\in\M\) and \(v\in\R^d\).
\end{lemma}
\begin{proof}
It clearly suffices to prove that \(F_n(x,v)\nearrow\rho(x,v)\) for any fixed
\(x\in\M\) and \(v\in\mathbb S^{d-1}\).\\
{\color{blue}\textsc{Case 1.}} Assume \(v\in V_x\coloneqq D(\rho_x)\).
We argue by contradiction: suppose there is \(t>0\) such that
\begin{equation}\label{eq:approx_subFinsl_1}
F_n(x,v)\leq\rho(x,v)-t\quad\text{ for every }n\in\N.
\end{equation}
Fix any distance \(\sfd\) on \(\M\) that induces the manifold topology.
Since the function \(\rho\) is lower semicontinuous by definition, we can
find \(r>0\) such that
\begin{equation}\label{eq:approx_subFinsl_2}
\rho(y,w)\geq\rho(x,v)-\frac{t}{2}\quad\text{ for every }(y,w)\in\M\times\R^d
\text{ with }\sfd(x,y),|v-w|<r.
\end{equation}
Choose any \(\bar n\in\N\) such that \(1/\bar n<\min\{r,t/4\}\)
and \(x\in\bigcap_{n\geq\bar n}G_n\), with \(G_n\) defined as in \eqref{eq:def_G_n}.
Fix any \(n\geq\bar n\). Item d) of Proposition \ref{prop:constr_g_n} grants the existence
of a point \(z\in S_n\cap B_{1/n}(x)\) and a vector
\(w_z\in V_z\cap\mathbb S^{d-1}\) such that \(F_n(x,v)\geq F_n(z,v)\)
and \(F_n(z,v-w_z)+|v-w_z|\leq 1/n\). Moreover, item c) of Proposition
\ref{prop:constr_g_n} ensures that \(F_n(z,w_z)\geq\rho(z,w_z)-1/n\).
All in all, we conclude that
\[
F_n(x,v)\geq F_n(z,v)\geq F_n(z,w_z)-\frac{1}{n}\geq\rho(z,w_z)-\frac{2}{n}
\overset{\eqref{eq:approx_subFinsl_2}}\geq\rho(x,v)-\frac{t}{2}-\frac{2}{n}
>\rho(x,v)-t,
\]
which is in contradiction with \eqref{eq:approx_subFinsl_1}. This proves
that \(F_n(x,v)\nearrow\rho(x,v)\) in the case \(v\in V_x\).\\
{\color{blue}\textsc{Case 2.}} Assume \(v\notin V_x\). Choose those elements
\(v'\in V_x\), \(w\in V_x^\perp\cap\mathbb S^{d-1}\), and \(\beta>0\)
for which \(v=v'+\beta\,w\). Fix any \(\bar n\in\N\) with
\(x\in\bigcap_{n\geq\bar n}G_n\). Then item b) of Proposition \ref{prop:constr_g_n}
yields \(F_n(x,w)\geq n\) for all \(n\geq\bar n\). Taking into account
item a) of Proposition \ref{prop:constr_g_n}, this gives
\[
F_n(x,v)\geq\beta\,F_n(x,w)-F_n(x,v')\geq\beta n-\rho(x,v')
\overset{n}\longrightarrow +\infty=\rho(x,v).
\]
Therefore, the statement is proven.
\end{proof}
\begin{theorem}[Approximation of generalised metrics]\label{thm:approx_subFinsl}
Let \(\M\) be a smooth manifold. Let \(\rho\) be a generalised metric on \(\M\).
Then there exists a sequence \((F_n)_n\) of Finsler metrics on \(\M\) such that
\begin{equation}\label{eq:conv_metrics_below}
F_n(x,v)\nearrow\rho(x,v)\quad\text{ for every }x\in\M\text{ and }v\in\T_x\M.
\end{equation}
\end{theorem}
\begin{proof}
Let us denote \(d\coloneqq\dim\M\). Since \(\M\) is paracompact (and second
countable), we can find a locally finite, open covering \((\M_i)_{i\in\N}\)
of \(\M\) such that the tangent bundle \(\T\M\) admits a local trivialisation
\(\psi_i\colon\T\M_i\to\M_i\times\R^d\) for every \(i\in\N\). Fix any
partition of unity \((\varphi_i)_i\subseteq C^\infty(\M)\) subordinated to
\((\M_i)_i\). Given any \(i\in\N\), we can apply Proposition \ref{prop:constr_g_n}
and Lemma \ref{lem:metrics_conv} to obtain a sequence \((\tilde F^i_n)_n\)
of Finsler metrics over \(\M_i\times\R^d\) such that
\(\tilde F^i_n(x,\tilde v)\nearrow(\rho\circ\psi_i^{-1})(x,\tilde v)\)
as \(n\to\infty\) for every \(x\in\M_i\) and \(\tilde v\in\R^d\).
Therefore, let us define \(F_n\colon\T\M\to[0,+\infty)\) as
\[
F_n(x,v)\coloneqq\sum_{i\in\N}\varphi_i(x)\,(\tilde F^i_n\circ\psi_i)(x,v)
\quad\text{ for every }x\in\M\text{ and }v\in\T_x\M.
\]
It can be readily checked that \((F_n)_n\) is a sequence of Finsler
metrics on \(\M\) satisfying \eqref{eq:conv_metrics_below}.
\end{proof}
\begin{remark}\label{rmk:approx_subRiem}{\rm
Under some additional assumptions, we can actually improve the statement of Theorem
\ref{thm:approx_subFinsl}: if we further suppose that \(\rho_x|_{D(\rho_x)}\)
is a Hilbert norm for every \(x\in\M\), then there exists a sequence \((g_n)_n\)
of Riemannian metrics on \(\M\) such that
\[
\sqrt{(g_n)_x(v,v)}\nearrow\rho(x,v)\quad\text{ for every }x\in\M
\text{ and }v\in\T_x\M.
\]
This fact can be proven by slightly modifying (actually, simplifying)
the arguments we discussed in the present section. More precisely,
it is sufficient to notice that in this case the norm \(\n\) defined in
\eqref{eq:def_n} is induced by a scalar product, and to omit
\textsc{Step 3} from the proof of Lemma \ref{lem:main_lemma}
(just defining \(\n\coloneqq\n''\)). Another proof of this
result can be found in \cite[Proof of Corollary 1.5]{LD10}.
\fr}\end{remark}
\section{Sub-Finsler manifolds}
\subsection{Definitions and main properties}
In this subsection we recall the notion of sub-Finsler manifold
and its main properties. The following material is taken
from \cite[Section 2.3]{LeDonne_lecture}.
\medskip

Given a smooth manifold \(\M\), we denote by \({\rm Vec}(\M)\)
the space of all smooth vector fields on \(\M\). Moreover, we define the map
\({\sf Der}\colon C\big([0,1],\M\big)\times[0,1]\to\T\M\) as
\begin{equation}\label{eq:Der}
{\sf Der}(\gamma,t)\coloneqq\left\{\begin{array}{ll}
(\gamma_t,\dot\gamma_t)\\
(\gamma_t,0)
\end{array}\quad\begin{array}{ll}
\text{ if }\dot\gamma_t\in\T_{\gamma_t}\M\text{ exists,}\\
\text{ otherwise.}
\end{array}\right.
\end{equation}
It is well-known that \(\sf Der\) is a Borel map.
For any \(v,w\in{\rm Vec}(\M)\), we denote by
\([v,w]\in{\rm Vec}(\M)\) the Lie brackets of \(v\) and \(w\).
Given any subset \(\mathcal F\) of \({\rm Vec}(\M)\),
we define the space \({\rm Lie}(\mathcal F)\subseteq{\rm Vec}(\M)\) as the Lie algebra
generated by the family \(\mathcal F\), i.e.,
\[{\rm Lie}(\mathcal F)\coloneqq
{\rm span}\Big\{[v_1,\ldots,[v_{j-1},v_j]\ldots]\;\Big|\;
j\in\N,\;v_1,\ldots,v_j\in\mathcal F\Big\}.\]
We set \({\rm Lie}_x(\mathcal F)\coloneqq\big\{v(x)\,:
\,v\in{\rm Lie}(\mathcal F)\big\}\leq\T_x\M$ for every \(x\in\M\).
We say that the family \(\mathcal F\) satisfies the \emph{H\"{o}rmander condition}
provided \({\rm Lie}_x(\mathcal F)=\T_x\M\) holds for every \(x\in\M\).
\begin{definition}[Sub-Finsler manifold]\label{def:sF_manifold}
Let \(\M\) be a smooth manifold. Then a triple \((\E,\sigma,\psi)\) is said to be
a \emph{sub-Finsler structure} on \(\M\) provided the following properties hold:
\begin{itemize}
\item[\(\rm i)\)] \(\E\) is a smooth vector bundle over \(\M\),
\item[\(\rm ii)\)] \(\sigma\) is a continuous metric over \(\E\),
\item[\(\rm iii)\)] \(\psi\colon\E\to\T\M\) is a morphism of smooth
vector bundles such that the family \(\mathcal D\) of
\emph{smooth horizontal vector fields} on \(\M\), which is defined as
\[
\mathcal D\coloneqq\big\{\psi\circ u\;\big|\;
u\text{ is a smooth section of }\E\big\}\subseteq{\rm Vec}(\M),
\]
satisfies the H\"{o}rmander condition.
\end{itemize}
The quadruple \((\M,\E,\sigma,\psi)\) is said to be a \emph{generalised sub-Finsler
manifold} (or just a \emph{sub-Finsler manifold}, for brevity).
If \((\E_x,\sigma|_{\E_x})\) is a Hilbert space
for every \(x\in\M\) and the family of squared norms \((\sigma|_{\E_x})^2\)
smoothly depends on \(x\), then \((\M,\E,\sigma,\psi)\) is called a
\emph{generalised sub-Riemannian manifold} (or just a \emph{sub-Riemannian manifold}).
\end{definition}

The family \(\mathcal D\) of smooth horizontal vector fields is a finitely-generated
module over \(C^\infty(\M)\).
The continuous distribution \(\{\mathcal D_x\}_{x\in\M}\) associated
with \((\M,\E,\sigma,\psi)\) is defined as
\[
\mathcal D_x\coloneqq\big\{v(x)\,\big|\,v\in\mathcal D\big\}\leq\T_x\M
\quad\text{ for every }x\in\M.
\]
We say that \(r(x)\coloneqq\dim \mathcal D_x\leq\dim\M$ is the \emph{rank}
of the sub-Finsler structure \((\E,\sigma,\psi)\) at \(x\in\M\).

Given any point \(x\in\M\) and any vector \(v\in\mathcal D_x\),
we define the quantity \({\|v\|}_x\) as
\begin{equation}\label{eq:horizontal_norm}
{\|v\|}_x\coloneqq\inf\big\{\sigma(u)\;\big|\;u\in\E_x,\,(x,v)=\psi(u)\big\}.
\end{equation}
Therefore, it holds that \({\|\cdot\|}_x\) is a norm on \(\mathcal D_x\).
Furthermore, if \((\E,\sigma,\psi)\) is a sub-Riemannian structure on \(\M\), then each
norm \({\|\cdot\|}_x\) is induced by some scalar product \(\la\cdot,\cdot\ra_x\).
\bigskip

\begin{definition}[Horizontal curve]\label{def:horizontal_curve}
Let \((\M,\E,\sigma,\psi)\) be a sub-Finsler manifold. Let \(\gamma\colon[0,1]\to\M\)
be a continuous curve such that for any \(\bar t\in[0,1]\)
there exist \(\delta>0\) and a chart \((U,\phi)\) of \(\M\) such that
\(\gamma(I)\subseteq U\) and \(\phi\circ\gamma|_I\colon I\to\R^{\dim\M}\)
is Lipschitz, where we set \(I\coloneqq(\bar t-\delta,\bar t+\delta)\cap[0,1]\).
Then the curve \(\gamma\) is said to be \emph{horizontal} provided
there is an \(L^\infty\)-section \(u\) of the pullback bundle \(\gamma^*\E\)
-- i.e., a map \([0,1]\ni t\mapsto u(t)\in\E_{\gamma_t}\) that is measurable
and essentially bounded -- such that
\[(\gamma_t,\dot\gamma_t)=\psi\big(u(t)\big)
\quad\text{ holds for a.e.\ }t\in[0,1].\]
The \emph{sub-Finsler length} of the curve \(\gamma\) is defined as
\(\ell_{\rm CC}(\gamma)\coloneqq\int_0^1{\|\dot\gamma_t\|}_{\gamma_t}\,\d t\).
\end{definition}
\begin{definition}[Carnot--Carath\'{e}odory distance]\label{def:CC_distance}
Let \((\M,\E,\sigma,\psi)\) be a sub-Finsler manifold. Fix any \(x,y\in\M\).
Then we define the \emph{Carnot--Carath\'{e}odory distance} between
\(x\) and \(y\) as
\begin{equation}\label{eq:d_CC}
\sfd_{\rm CC}(x,y)\coloneqq\inf\Big\{\ell_{\rm CC}(\gamma)\;\Big|\;
\gamma\text{ is a horizontal curve in }\M\text{ such that }
\gamma_0=x\text{ and }\gamma_1=y\Big\}.
\end{equation}
\end{definition}
\begin{theorem}[Chow--Rashevskii]\label{thm:Chow-Rashevskii}
Let \((\M,\E,\sigma,\psi)\) be a sub-Finsler manifold.
Then \(\sfd_{\rm CC}\) is a distance on \(\M\) that
induces the manifold topology.
\end{theorem}
The metric space \((\M,\sfd_{\rm CC})\) is complete if and only
if \(\bar B_r(x)\) is compact for all \(x\in\M\) and \(r>0\).
\begin{proposition}\label{prop:horizontal_vs_Lipschitz}
Let \((\M,\E,\sigma,\psi)\) be a sub-Finsler manifold. Let \(\gamma\colon[0,1]\to\M\)
be a curve in \(\M\). Then \(\gamma\) is \(\sfd_{\rm CC}\)-Lipschitz
if and only if it is horizontal. Moreover, in such case it holds that
\[{\|\dot\gamma_t\|}_{\gamma_t}=
\lim_{h\to 0}\frac{\sfd_{\rm CC}(\gamma_{t+h},\gamma_t)}{|h|}
\quad\text{ for a.e.\ }t\in[0,1].\]
\end{proposition}
\subsection{Structure of the horizontal bundle}
Let \((\M,\E,\sigma,\psi)\) be a sub-Finsler manifold, whose associated
distribution is denoted by ${\{\mathcal D_x\}}_{x\in\M}$.
Then we define the \emph{horizontal bundle} \(\H\M\) as
\[\H\M\coloneqq\bigsqcup_{x\in\M}\mathcal D_x.\]
Moreover, we define the function \(\rho\colon\T\M\to[0,+\infty]\) as
\begin{equation}\label{eq:Norm}
\rho(x,v)\coloneqq\left\{\begin{array}{ll}
{\|v\|}_x\\
+\infty
\end{array}\quad\begin{array}{ll}
\text{ if }(x,v)\in\H\M,\\
\text{ otherwise.}
\end{array}\right.
\end{equation}
\begin{lemma}\label{lem:Norm_Borel}
Let \((\M,\E,\sigma,\psi)\) be a sub-Finsler manifold.
Then the function \(\rho\) defined in \eqref{eq:Norm} is a generalised
metric on \(\M\). In particular, the horizontal bundle \(\H\M\) is a
Borel subset of \(\T\M\).
\end{lemma}
\begin{proof}
First of all, observe that \(x\mapsto D(\rho_x)=\mathcal D_x\)
is a continuous distribution by Theorem \ref{thm:rank_vary_distr}.
With this said, we only have to prove that the function \(\rho\) is lower semicontinuous.
To this aim, let us fix a sequence
\(\big\{(x_n,v_n)\big\}_{n\in\N\cup\{\infty\}}\subseteq\T\M\)
such that \((x_n,v_n)\to(x_\infty,v_\infty)\). We claim that
\begin{equation}\label{eq:Norm_Borel_claim}
\rho(x_\infty,v_\infty)\leq\varliminf_{n\to\infty}\rho(x_n,v_n).
\end{equation}
Without loss of generality, we can assume that
\(\varliminf_n\rho(x_n,v_n)<+\infty\). For any \(n\in\N\) choose an
element \(u_n\in\E_{x_n}\) such that \(\psi(u_n)=(x_n,v_n)\) and
\(\sigma(u_n)\leq\rho(x_n,v_n)+1/n\). Moreover, pick a subsequence \((n_k)_k\)
with \(\lim_k\rho(x_{n_k},v_{n_k})=\varliminf_n\rho(x_n,v_n)\).
Therefore, it holds that \(\big\{\sigma(u_{n_k})\big\}_{k\in\N}\) is bounded,
thus there exists \(u_\infty\in\E_{x_\infty}\) such that (possibly passing
to a not relabelled subsequence) we have \(u_{n_k}\to u_\infty\). Note that
\(\psi(u_\infty)=\lim_k\psi(u_{n_k})=\lim_k(x_{n_k},v_{n_k})=(x_\infty,v_\infty)\)
by continuity of \(\psi\). Consequently, by using the continuity of \(\sigma\)
and the definition of \(\rho_{x_\infty}\) we conclude that
\[
\rho(x_\infty,v_\infty)\leq\sigma(u_\infty)=\lim_{k\to\infty}\sigma(u_{n_k})
\leq\lim_{k\to\infty}\rho(x_{n_k},v_{n_k})+\frac{1}{n_k}
=\varliminf_{n\to\infty}\rho(x_n,v_n),
\]
which proves the claim \eqref{eq:Norm_Borel_claim}. Hence, the statement is finally achieved.
\end{proof}

A vector field \(v\colon\M\to\T\M\) is said to be a \emph{section} of \(\H\M\)
provided \(v(x)\in\mathcal D_x\) for every \(x\in\M\).
We say that a section \(v\) of \(\H\M\) is
\emph{Borel} provided it is Borel measurable as a map from \(\M\) to \(\T\M\).

It immediately follows from Lemma \ref{lem:Norm_Borel} that
\[\M\ni x\longmapsto{\big\|v(x)\big\|}_x\in\R\text{ is a Borel function,}
\quad\text{ for every Borel section }v\text{ of }\H\M.\]
The space of Borel sections of \(\H\M\) is a vector space with
respect to the usual pointwise operations.
\begin{definition}[The space \(L^2(\H\M;\mu)\)]\label{def:sect_HM}
Let \((\M,\E,\sigma,\psi)\) be a sub-Finsler manifold.
Let \(\mu\) be a non-negative Borel measure on \((\M,\sfd_{\rm CC})\).
Then we define the space \(L^2(\H\M;\mu)\) as the set of (equivalence classes up
to \(\mu\)-a.e.\ equality of) all Borel sections \(v\) of the horizontal bundle
\(\H\M\) such that \(\M\ni x\mapsto{\big\|v(x)\big\|}_x\in\R\) belongs to \(L^2(\mu)\).
The space \(L^2(\H\M;\mu)\) is an \(L^\infty(\mu)\)-module
with respect to the natural pointwise operations, thus in particular
it is a vector space.
\end{definition}
\begin{remark}[Pointwise parallelogram identity]\label{rmk:geom_pr}{\rm
Let \((\M,\E,\sigma,\psi)\) be a sub-Riemannian manifold.
Given that each space \(\big(\mathcal D_x,{\|\cdot\|}_x\big)\) is Hilbert,
we readily deduce that
\[
{\big\|v(x)+w(x)\big\|}^2_x+{\big\|v(x)-w(x)\big\|}^2_x
=2\,{\big\|v(x)\big\|}^2_x+2\,{\big\|w(x)\big\|}^2_x
\quad\text{ holds for }\mu\text{-a.e.\ }x\in\M,
\]
for every \(v,w\in L^2(\H\M;\mu)\).
\fr}\end{remark}
\medskip

Given any smooth function \(f\in C^\infty(\M)\), we denote by \(\d f\) its
differential, which is a smooth section of the cotangent bundle \(\T^*\M\).
Then the \emph{horizontal differential} \(\d_\H f\) of \(f\) is defined as
\begin{equation}\label{eq:horizontal_differential}
\d_\H f(x)\coloneqq\d_x f|_{\mathcal D_x}\in\mathcal D_x^*
\quad\text{ for every }x\in\M.
\end{equation}
\begin{lemma}\label{lem:grad_fiberwise_dense}
Let \((\M,\E,\sigma,\psi)\) be a given sub-Finsler manifold.
Then there exists a countable family of functions \(\mathscr C\subseteq C^\infty_c(\M)\)
such that \(\big\{\d_\H f(x)\,:\,f\in\mathscr C\big\}\) is dense in
\(\mathcal D_x^*\) for every \(x\in\M\).
\end{lemma}
\begin{proof}
Call \(n\coloneqq\dim\M\). By Lindel\"{o}f lemma we know that there
exists an open covering \((\Omega_j)_{j\in\N}\) of \(\M\) with the following
property: for every \(j\in\N\) there exist some functions
\(f^j_1,\ldots,f^j_n\in C^\infty_c(\M)\) such that
\(\d_x f^j_1,\ldots,\d_x f^j_n\) is a basis of \(\T^*_x\M\) for every \(x\in\Omega_j\).
Consequently, \(\d_\H f^j_1(x),\ldots,\d_\H f^j_n(x)\) generate
\(\mathcal D_x^*\) for every \(x\in\Omega_j\). Calling \({\rm V}_j\) the
\(\mathbb Q\)-linear subspace of \(C^\infty_c(\M)\) generated by
\(f^j_1,\ldots,f^j_n\), one clearly has that
\(\big\{\d_\H f(x)\,:\,f\in{\rm V}_j\big\}\) is dense in \(\mathcal D_x^*\)
for every \(x\in\Omega_j\). Therefore, the countable family of functions
\(\mathscr C\coloneqq\bigcup_{j\in\N}{\rm V}_j\)
fulfills the required properties.
\end{proof}
\begin{lemma}
Let \((\M,\E,\sigma,\psi)\) be a sub-Finsler manifold.
Let \(f\in C^1_c(\M)\). Then \(f\in\LIP(\M)\) and
\begin{equation}\label{eq:estimate_nabla_H}
{\big\|\d_{\rm H}f(x)\big\|}_x^*\leq\lip(f)(x)\quad\text{ for every }x\in\M.
\end{equation}
\end{lemma}
\begin{proof}
Lipschitzianity of \(f\) can be proven by arguing, e.g., as in
\cite[Lemma 3.16]{AgrBarBos19}. To show \eqref{eq:estimate_nabla_H},
let \(v\in\mathcal D_x\) and \(\eps>0\) be fixed.
We know from \eqref{eq:horizontal_norm} that there exists \(u\in\E_x\)
such that \((x,v)=\psi(u)\) and \(\sigma(u)\leq{\|v\|}_x+\eps\). Choose a
smooth section \(\eta\) of \(\E\) such that \(\eta(x)=u\).
Since \(\psi\circ\eta\) is a smooth vector field on \(\M\),
there exists a smooth solution \(\gamma\colon[0,\delta']\to\M\) to the ODE
\[\left\{\begin{array}{ll}
\dot\gamma_t=(\psi\circ\eta)(\gamma_t)\quad\text{ for every }t\in[0,\delta'],\\
\gamma_0=x.
\end{array}\right.\]
Being \(\eta\circ\gamma\) continuous, we can find \(\delta\in(0,\delta')\) such
that \(\sigma\big(\eta(\gamma_t)\big)\leq\sigma(u)+\eps\) for every \(t\in[0,\delta]\).
Moreover, again by
\eqref{eq:horizontal_norm} we have that
\({\|\dot\gamma_t\|}_{\gamma_t}\leq\sigma\big(\eta(\gamma_t)\big)\)
for all \(t\in[0,\delta]\). Combining the previous estimates we get that
\({\|\dot\gamma_t\|}_{\gamma_t}\leq{\|v\|}_x+2\,\eps\) for every 
\(t\in[0,\delta]\). Therefore, we conclude that
\[\begin{split}
\d_\H f(x)[v]&\overset{\phantom{\eqref{eq:d_CC}}}=
\d_x f(\dot\gamma_0)=\lim_{t\searrow 0}\frac{f(\gamma_t)-f(x)}{t}
\leq\lip(f)(x)\,\lim_{t\searrow 0}\frac{\sfd_{\rm CC}(\gamma_t,\gamma_0)}{t}\\
&\overset{\eqref{eq:d_CC}}\leq
\lip(f)(x)\,\lim_{t\searrow 0}\frac{1}{t}\int_0^t{\|\dot\gamma_s\|}_{\gamma_s}\,\d s
\leq\lip(f)(x)\,\big({\|v\|}_x+2\,\eps\big).
\end{split}\]
Letting \(\eps\searrow 0\) we see that
\(\d_\H f(x)[v]\leq\lip(f)(x)\,{\|v\|}_x\)
for all \(v\in\mathcal D_x\), whence \eqref{eq:estimate_nabla_H} follows.
\end{proof}
\section{Main result: infinitesimal Hilbertianity of sub-Riemannian manifolds}
\subsection{Derivations on weighted sub-Finsler manifolds}
The aim of this subsection is to provide an alternative to the
representation formula \eqref{eq:ptwse_norm_b} for the pointwise norm of
a derivation (with divergence) over a weighted sub-Finsler manifold
\(\M\). We would like
to express the pointwise norm of a derivation \(\b\) as the essential
supremum of the functions \(\b(f)\), where \(f\) varies in a countable
family of \(1\)-Lipschitz smooth functions. Given that the distance 
functions \(\sfd_{\rm CC}(\cdot,\bar x)\) from fixed points \(\bar x\in\M\) are not smooth (thus in particular not almost everywhere differentiable with respect to an arbitrary measure on \(\M\)), a new representation formula is needed.
\medskip

The following result states that any Carnot--Carath\'{e}odory distance
can be (monotonically) approximated by distances associated to
suitable Finsler metrics. A word on notation: given a Finsler metric \(F\)
on a manifold \(\M\), we denote by \(\sfd_F\) the induced distance on \(\M\).
\begin{theorem}\label{thm:approx_sR_with_Riem}
Let \((\M,\E,\sigma,\psi)\) be a sub-Finsler manifold.
Then there exists a sequence \((F_n)_n\) of Finsler metrics on \(\M\)
such that \(\sfd_{F_n}(x,y)\nearrow\sfd_{\rm CC}(x,y)\) holds
for every \(x,y\in\M\).
\end{theorem}
\begin{proof}
Define \(\rho\) as in \eqref{eq:Norm} and consider a sequence \((F_n)_n\)
of approximating Finsler metrics as in Theorem \ref{thm:approx_subFinsl}.
Let \(x,y\in\M\) be fixed. Let \(\gamma\colon[0,1]\to\M\) be a curve
joining \(x\) and \(y\) that is Lipschitz when read in charts (i.e.,
as in Definition \ref{def:horizontal_curve}). Calling \(\ell_{F_n}\) the
length functional associated to \(F_n\), it holds that
\(\ell_{F_n}(\gamma)\leq\ell_{F_{n+1}}(\gamma)\leq\ell_{\rm CC}(\gamma)\)
for every \(n\in\N\), thus by taking the infimum over \(\gamma\) we deduce
that \(\sfd_{F_n}(x,y)\leq\sfd_{F_{n+1}}(x,y)\leq\sfd_{\rm CC}(x,y)\).
Given any \(n\in\N\), by definition of \(\sfd_{F_n}\)
we find a constant-speed Lipschitz curve \(\gamma^n\colon[0,1]\to\M\),
where the target is endowed with \(\sfd_{F_n}\), such that
\begin{equation}\label{eq:choice_gamma^n}
\ell_{F_n}(\gamma^n)\leq\sfd_{F_n}(x,y)+\frac{1}{n}.
\end{equation}
Fix \(n\in\N\). The above considerations yield
\[
\ell_{F_n}(\gamma^i)\leq\ell_{F_i}(\gamma^i)\overset{\eqref{eq:choice_gamma^n}}\leq
\sfd_{F_i}(x,y)+\frac{1}{i}\leq\sfd_{\rm CC}(x,y)+1\quad\text{ for every }i\geq n.
\]
This shows that \((\gamma^i)_{i\geq n}\) is an equiLipschitz family
of curves (with respect to \(\sfd_{F_n}\)). By combining Arzel\`{a}--Ascoli theorem
with a diagonalisation argument, we thus obtain a curve \(\gamma\colon[0,1]\to\M\),
which is Lipschitz with respect to each distance \(\sfd_{F_n}\), such that
(up to a not relabelled subsequence)
\begin{equation}\label{eq:unif_conv_gamma}
\lim_{i\to\infty}\sup_{t\in[0,1]}\sfd_{F_n}(\gamma^i_t,\gamma_t)=0
\quad\text{ for every }n\in\N.
\end{equation}
Since \(\ell_{F_n}\) is lower semicontinuous under
uniform convergence of curves, we deduce from \eqref{eq:unif_conv_gamma} that
\begin{equation}\label{eq:ineq_ell_F_n}
\ell_{F_n}(\gamma)\leq\varliminf_{i\to\infty}\ell_{F_n}(\gamma^i)
\leq\varliminf_{i\to\infty}\ell_{F_i}(\gamma^i)
\overset{\eqref{eq:choice_gamma^n}}\leq\lim_{i\to\infty}\sfd_{F_i}(x,y)
\quad\text{ for every }n\in\N.
\end{equation}
Therefore, by using the monotone convergence theorem we obtain that
\[
\int_0^1\rho(\gamma_t,\dot\gamma_t)\,\d t
=\lim_{n\to\infty}\int_0^1 F_n(\gamma_t,\dot\gamma_t)\,\d t
=\lim_{n\to\infty}\ell_{F_n}(\gamma)
\overset{\eqref{eq:ineq_ell_F_n}}\leq\lim_{i\to\infty}\sfd_{F_i}(x,y)
\leq\sfd_{\rm CC}(x,y)<+\infty,
\]
which implies that the curve \(\gamma\) is horizontal and satisfies
\(\ell_{\rm CC}(\gamma)=\sfd_{\rm CC}(x,y)=\lim_n\sfd_{F_n}(x,y)\).
\end{proof}
Although not strictly needed for our purposes, let us point
out an immediate well-known consequence (already proven in \cite{LD10})
of the previous theorem.
\begin{corollary}\label{cor:approx_sR_with_Riem}
Let \((\M,\E,\sigma,\psi)\) be a sub-Riemannian manifold.
Then there exists a sequence \((g_n)_n\) of Riemannian metrics on \(\M\)
such that \(\sfd_{g_n}(x,y)\nearrow\sfd_{\rm CC}(x,y)\) holds
for every \(x,y\in\M\).
\end{corollary}
\begin{proof}
It follows from Theorem \ref{thm:approx_sR_with_Riem} by taking
Remark \ref{rmk:approx_subRiem} into account.
\end{proof}

We shall also need the ensuing approximation result for real-valued Lipschitz
functions that are defined on a Finsler manifold.
\begin{lemma}\label{lem:approx_Lip_with_Cinfty}
Let \((\M,F)\) be a Finsler manifold. Let \(f\in\LIP_{bs}(\M)\) be given.
Then there exists a sequence \((f_n)_n\subseteq C^1_{bs}(\M)\) with
\(\sup_n\Lip(f_n)\leq\Lip(f)\) such that \(f_n\to f\) uniformly on \(\M\).
\end{lemma}
\begin{proof}
We call \(C\coloneqq\max_\M|f|\). We know (for instance,
from \cite[Theorem 2.6]{LP19}) that for any \(n\in\N\) there
exists a function \(g_n\in C^1_{bs}(\M)\) such that \(\Lip(g_n)\leq\Lip(f)+1/n\)
and \(|g_n-f|\leq 1/n\) on \(\M\). Set \(c_n\coloneqq\Lip(f)/\big(\Lip(f)+1/n\big)\)
and \(f_n\coloneqq c_n\,g_n\). Therefore \(f_n\in C^1_{bs}(\M)\),
\(\Lip(f_n)\leq\Lip(f)\), and
\[|f_n-f|\leq|f_n-g_n|+|g_n-f|=(1-c_n)|g_n|+|g_n-f|
\leq\frac{|f|+1/n}{n\,\Lip(f)+1}+\frac{1}{n}
\leq\frac{C+1}{n\,\Lip(f)+1}+\frac{1}{n},\]
thus accordingly \(f_n\to f\) uniformly on \(\M\), as required.
\end{proof}

We are now in a position to prove a representation formula
for the pointwise norm of derivations on weighted sub-Finsler manifolds,
by combining the above two results with Proposition \ref{prop:ptwse_norm_b}.
\begin{theorem}\label{thm:ptwse_norm_b_bis}
Let \((\M,\E,\sigma,\psi)\) be a sub-Finsler manifold such that \(\sfd_{\rm CC}\) is complete.
Let \(\mu\geq 0\) be a finite Borel measure on \((\M,\sfd_{\rm CC})\).
Then there exists a countable family \(\mathscr F\subseteq C^1_c(\M)\cap\LIP(\M)\)
such that \(\Lip(f)\leq 1\) for every \(f\in\mathscr F\) and
\[|\b|=\underset{f\in\mathscr F}{\rm ess\,sup\,}\b(f)\;\;\;\mu\text{-a.e.}
\quad\text{ for every }\b\in{\rm Der}^{2,2}(\M;\mu).\]
\end{theorem}
\begin{proof}
Fix a dense sequence \((x_k)_k\subseteq\M\). Theorem \ref{thm:approx_sR_with_Riem}
grants the existence of a sequence \((F_i)_i\) of Finsler metrics on \(\M\) such
that \(\sfd_{F_i}\nearrow\sfd_{\rm CC}\) pointwise on \(\M\times\M\).
Choose a family \(\{\eta_{jk}\}_{j,k\in\N}\) of cut-off functions with these properties:
given \(j,k\in\N\), we have that \(\eta_{jk}\colon\M\to[0,1-1/j]\) is a
boundedly-supported Lipschitz function (with respect to \(\sfd_{F_1}\))
such that \(\eta_{jk}=1-1/j\) on \(B_j^{\sfd_{\rm CC}}(x_k)\) and
\(\Lip^{\sfd_{F_1}}(\eta_{jk})\leq 1/j^2\). Observe that for any
\(i,j,k\in\N\) it holds that
\[\begin{split}
\Lip^{\sfd_{F_i}}\big((\sfd_{F_i}(\cdot,x_k)\wedge j)\,\eta_{jk}\big)&\leq
\Lip^{\sfd_{F_i}}\big(\sfd_{F_i}(\cdot,x_k)\wedge j\big)\,\max_\M|\eta_{jk}|+
\Lip^{\sfd_{F_i}}(\eta_{jk})\max_{\M}\big|\sfd_{F_i}(\cdot,x_k)\wedge j\big|\\
&\leq 1-\frac{1}{j}+j\,\Lip^{\sfd_{F_1}}(\eta_{jk})\leq 1.
\end{split}\]
Therefore, Lemma \ref{lem:approx_Lip_with_Cinfty}
guarantees the existence of a function
\(f_{ijk}\in C^1_c(\M)\cap\LIP^{\sfd_{F_i}}(\M)\)
that satisfies \(\Lip^{\sfd_{F_i}}(f_{ijk})\leq 1\) and
\[
\big|(\sfd_{F_i}(x,x_k)\wedge j)\,\eta_{jk}(x)-f_{ijk}(x)\big|\leq\frac{1}{i}
\quad\text{ for every }x\in\M.
\]
Define \(\mathscr F\coloneqq\{f_{ijk}\,:\,i,j,k\in\N\}\).
Note that \(\mathscr F\subseteq C^1_c(\M)\cap\LIP^{\sfd_{\rm CC}}(\M)\)
and \(\sup_{f\in\mathscr F}\Lip^{\sfd_{\rm CC}}(f)\leq 1\), as
\(\Lip^{\sfd_{\rm CC}}(f_{ijk})\leq\Lip^{\sfd_{F_i}}(f_{ijk})\leq 1\)
for all \(i,j,k\in\N\). Now let us call
\begin{equation}\label{eq:ptwse_norm_b_bis}
{\sf n}(\b)\coloneqq\underset{f\in\mathscr F}{\rm ess\,sup\,}\b(f)
\quad\text{ for every }\b\in{\rm Der}^{2,2}(\M;\mu).
\end{equation}
Since \(\b(f)\leq|\b|\,\Lip^{\sfd_{\rm CC}}(f)\leq|\b|\) holds \(\mu\)-a.e.\ for all
\(f\in\mathscr F\), we deduce that \({\sf n}(\b)\leq|\b|\) holds \(\mu\)-a.e.\ as well.
To prove the converse inequality, fix \(\eps>0\). Proposition \ref{prop:ptwse_norm_b}
grants the existence of a Borel partition \((A_{jk})_{j,k}\)
of \(\M\) such that
\(\sum_{j,k}\nchi_{A_{jk}}\,\b\big((\sfd_{\rm CC}(\cdot,x_k)\wedge j)\,\eta_{jk}\big)\geq|\b|-\eps\)
in the \(\mu\)-a.e.\ sense. Fix \(j,k\in\N\) and choose a sequence
\((\varphi_n)_n\subseteq\LIP^{\sfd_{\rm CC}}_{bs}(\M)\) with \(\varphi_n\geq 0\)
converging to \(\nchi_{A_{jk}}\) strongly in \(L^2(\mu)\).
Given that \(\lim_i f_{ijk}(x)=(\sfd_{\rm CC}(x,x_k)\wedge j)\,\eta_{jk}(x)\) for
every \(x\in\M\), Lemma \ref{lem:der_continuity} yields
\[\int\varphi_n\,\b\big((\sfd_{\rm CC}(\cdot,x_k)\wedge j)\,\eta_{jk}\big)\,\d\mu=
\lim_{i\to\infty}\int\varphi_n\,\b(f_{ijk})\,\d\mu
\overset{\eqref{eq:ptwse_norm_b_bis}}\leq\int\varphi_n\,{\sf n}(\b)\,\d\mu
\quad\text{ for every }n\in\N,\]
thus by letting \(n\to\infty\) we deduce that
\[\int_{A_{jk}}|\b|\,\d\mu-\eps\,\mu(A_{jk})\leq
\int_{A_{jk}}\b\big((\sfd_{\rm CC}(\cdot,x_k)\wedge j)\,\eta_{jk}\big)\,\d\mu
\leq\int_{A_{jk}}{\sf n}(\b)\,\d\mu.\]
By summing over \(j,k\in\N\) we get that
\(\int|\b|\,\d\mu-\eps\,\mu(\M)\leq\int{\sf n}(\b)\,\d\mu\).
By letting \(\eps\searrow 0\) we finally conclude that
\(\int|\b|\,\d\mu\leq\int{\sf n}(\b)\,\d\mu\), which forces
the \(\mu\)-a.e.\ equality \(|\b|={\sf n}(\b)\), as desired.
\end{proof}
\subsection{Embedding theorem and its consequences}
This subsection is devoted to our main result, namely
Theorem \ref{thm:embedding_intro},
which states that the space of derivations \({\rm Der}^{2,2}(\M;\mu)\)
associated with a weighted sub-Finsler manifold \(\M\)
can be isometrically embedded into the space \(L^2(\H\M;\mu)\)
of all `geometric' \(2\)-integrable sections of the horizontal
bundle \(\H\M\). For the reader's convenience, we also recall here the statement.
\begin{theorem}[Embedding theorem]\label{thm:embedding}
Let \((\M,\E,\sigma,\psi)\) be a sub-Finsler manifold with \(\sfd_{\rm CC}\) complete.
Let \(\mu\) be a finite, non-negative Borel measure on \((\M,\sfd_{\rm CC})\).
Then there exists a unique linear operator
\({\rm I}\colon{\rm Der}^{2,2}(\M;\mu)\to L^2(\H\M;\mu)\) such that
\begin{equation}\label{eq:I(b)}
\d_\H f(x)\big[{\rm I}(\b)(x)\big]=\b(f)(x)
\quad\text{ holds for }\mu\text{-a.e.\ }x\in\M,
\end{equation}
for every \(\b\in{\rm Der}^{2,2}(\M;\mu)\) and \(f\in C^1_c(\M)\cap\LIP(\M)\).
Moreover, the operator \(\rm I\) satisfies
\begin{equation}\label{eq:I_isometry}
{\big\|{\rm I}(\b)(x)\big\|}_x=|\b|(x)\quad\text{ for }\mu\text{-a.e.\ }x\in\M,
\end{equation}
for every \(\b\in{\rm Der}^{2,2}(\M;\mu)\).
\end{theorem}
\begin{proof} We divide the proof into several steps:\\
{\color{blue}\textsc{Borel regularity.}} We aim to prove that any section
\({\rm I}(\b)\) of \(\H\M\) satisfying \eqref{eq:I(b)} is (equivalent to)
a Borel section. Given any \(\bar x\in\M\), we can find an open neighbourhood
\(\Omega\) of \(\bar x\) and some functions \(f_1,\ldots,f_n\in C^\infty_c(\M)\)
such that \(\d_x f_1,\ldots,\d_x f_n\) is a basis of \(\T_x^*\M\) for
all \(x\in\Omega\). Since each function
\(\Omega\ni x\mapsto\d_x f_i\big[{\rm I}(\b)(x)\big]\) is equivalent
to a Borel function by \eqref{eq:horizontal_differential} and
\eqref{eq:I(b)}, we deduce that \({\rm I}(\b)\) is equivalent to a Borel section
of \(\H\M\) on \(\Omega\). By Lindel\"{o}f lemma we can cover \(\M\) with
countably many sets \(\Omega\) with this property, whence \({\rm I}(\b)\) is
equivalent to a Borel section of \(\H\M\).\\
{\color{blue}\textsc{Integrability.}} The property \eqref{eq:I_isometry}
ensures that each Borel section \({\rm I}(\b)\) belongs to \(L^2(\H\M;\mu)\),
since \(|\b|\in L^2(\mu)\) by assumption.\\
{\color{blue}\textsc{Uniqueness.}} Fix \(\b\in{\rm Der}^{2,2}(\M;\mu)\) and
pick \(\mathscr C\subseteq C^1_c(\M)\cap\LIP(\M)\)
as in Lemma \ref{lem:grad_fiberwise_dense}. Then by writing \eqref{eq:I(b)}
for every function \(f\in\mathscr C\) we deduce that the element
\({\rm I}(\b)(x)\in\mathcal D_x\) is uniquely determined for
\(\mu\)-a.e.\ \(x\in\M\), thus the operator \(\rm I\) is unique.\\
{\color{blue}\textsc{Linearity.}} Let \(\b,\b'\in{\rm Der}^{2,2}(\M;\mu)\)
and \(\lambda,\lambda'\in\R\) be given. Then \eqref{eq:I(b)} ensures that
\[\begin{split}
\d_\H f(x)\big[\lambda\,{\rm I}(\b)(x)+\lambda'\,{\rm I}(\b')(x)\big]
&=\lambda\,\d_\H f(x)\big[{\rm I}(\b)(x)\big]
+\lambda'\,\d_\H f(x)\big[{\rm I}(\b')(x)\big]\\
&=\lambda\,\b(f)(x)+\lambda'\,\b'(f)(x)
=(\lambda\,\b+\lambda'\,\b')(f)(x)
\end{split}\]
for every \(f\in C^1_c(\M)\cap\LIP(\M)\) and \(\mu\)-a.e.\ \(x\in\M\).
By uniqueness, we conclude that \(\rm I\) is linear.\\
{\color{blue}\textsc{Existence.}} Let \(\b\in{\rm Der}^{2,2}(\M;\mu)\) be fixed.
Consider its associated measure \(\ppi\) as in Theorem
\ref{thm:superposition_principle}. Define the measure
\(\hat\ppi\coloneqq\ppi\otimes\mathcal L_1\) on 
\(C\big([0,1],\M\big)\times[0,1]\), where \(\mathcal L_1\) stands for the
restriction of the Lebesgue measure \(\mathcal L^1\) to the interval \([0,1]\).
The evaluation map \({\rm e}\colon C\big([0,1],\M\big)\times[0,1]\to\M\), i.e.,
\[{\rm e}(\gamma,t)\coloneqq\gamma_t\quad\text{ for every }
\gamma\in C\big([0,1],\M\big)\text{ and }t\in[0,1],\]
is continuous. Therefore, it makes sense to consider the finite Borel measure
\(\nu\coloneqq{\rm e}_*\hat\ppi\) on \(\M\). An application of the
disintegration theorem \cite[Theorem 5.3.1]{AmbrosioGigliSavare08} provides us
with a weakly measurable family \(\{\hat\ppi_x\}_{x\in\M}\) of Borel probability
measures on \(C\big([0,1],\M\big)\times[0,1]\) such that
\begin{subequations}\begin{align}
\hat\ppi_x\text{ is concentrated on }{\rm e}^{-1}(\{x\})&
\quad\text{ for }\nu\text{-a.e.\ }x\in\M,\label{eq:disint_1}\\
\int\Psi(\gamma,t)\,\d\hat\ppi(\gamma,t)=
\int\bigg(\int\Psi(\gamma,t)\,\d\hat\ppi_x(\gamma,t)\bigg)\d\nu(x)&
\quad\text{ for every }\Psi\in L^1(\hat\ppi).\label{eq:disint_2}
\end{align}\end{subequations}
Since \(\ppi\)-a.e.\ curve \(\gamma\) is horizontal by Proposition
\ref{prop:horizontal_vs_Lipschitz}, we have that \(\dot\gamma_t\in\mathcal D_{\gamma_t}\)
holds for \(\hat\ppi\)-a.e.\ \((\gamma,t)\) by Fubini theorem. Consider
the Borel map \({\sf Der}\colon C\big([0,1],\M\big)\times[0,1]\to\T\M\) defined
in \eqref{eq:Der}. Then we know from \eqref{eq:disint_1} that for
\(\nu\)-a.e.\ \(x\in\M\) the measure
\(\mathfrak n_x\coloneqq{\sf Der}_*\hat\ppi_x\) can be viewed as a
Borel probability measure on \(\mathcal D_x\). Therefore, for any function
\(g\in\LIP_{bs}(\M)\) we have that
\begin{equation}\label{eq:superposition_equiv_2}\begin{split}
\int g\,|\b|\,\d\mu&\overset{\eqref{eq:superposition_principle_2}}=
\int\!\!\!\int_0^1 g(\gamma_t)\,|\dot\gamma_t|\,\d t\,\d\ppi(\gamma)
=\int g\big({\rm e}(\gamma,t)\big)\,
\rho\big({\sf Der}(\gamma,t)\big)\,\d\hat\ppi(\gamma,t)\\
&\overset{\eqref{eq:disint_2}}=
\int\bigg(\int g\big({\rm e}(\gamma,t)\big)\,
\rho\big({\sf Der}(\gamma,t)\big)\,\d\hat\ppi_x(\gamma,t)\bigg)\,\d\nu(x)\\
&\overset{\phantom{\eqref{eq:superposition_principle_2}}}=
\int g(x)\,\bigg(\int_{\mathcal D_x}{\|v\|}_x\,\d\mathfrak n_x(v)\bigg)\,\d\nu(x),
\end{split}\end{equation}
where the function \(\rho\) is defined as in \eqref{eq:Norm}.
Let us set \(\Phi(x)\coloneqq\int_{\mathcal D_x}{\|v\|}_x\,\d\mathfrak n_x(v)\)
for \(\nu\)-a.e.\ \(x\in\M\). The measurability of \(\Phi\) is granted by the
fact that \(\{\hat\ppi_x\}_{x\in\M}\) is a weakly measurable family of measures.
Moreover, by the arbitrariness of \(g\in\LIP_{bs}(\M)\) we deduce from
\eqref{eq:superposition_equiv_2} that \(|\b|\mu=\Phi\nu\). In particular
\(\Phi\in L^1(\nu)\), which implies that
\begin{equation}\label{eq:Phi_finite_ae}
\int_{\mathcal D_x}{\|v\|}_x\,\d\mathfrak n_x(v)<+\infty
\quad\text{ for }\nu\text{-a.e.\ }x\in\M.
\end{equation}
Given that \(\ppi\) is concentrated on non-constant Lipschitz curves having constant
speed, we also have that \(\dot\gamma_t\neq 0\) for \(\hat\ppi\)-a.e.\ \((\gamma,t)\),
or equivalently that \(\rho\circ{\sf Der}>0\) in the \(\hat\ppi\)-a.e.\ sense.
Hence \eqref{eq:superposition_equiv_2} ensures that
\(\Phi(x)=\int\rho\circ{\sf Der}\,\d\hat\pi_x>0\) holds for \(\nu\)-a.e.\ point
\(x\in\M\), which together with the identity \(|\b|\mu=\Phi\nu\) imply that \(\nu\ll\mu\).
The Bochner integral \(\int_{\mathcal D_x}v\,\d\mathfrak n_x(v)\) is
well-posed for \(\nu\)-a.e.\ point \(x\in\M\) by \eqref{eq:Phi_finite_ae},
therefore it makes sense to define
\[{\rm I}(\b)(x)\coloneqq\frac{\d\nu}{\d\mu}(x)\int_{\mathcal D_x}v\,\d\mathfrak n_x(v)
\in\mathcal D_x\quad\text{ for }\mu\text{-a.e.\ }x\in\M,\]
where \(\frac{\d\nu}{\d\mu}\) stands for the Radon--Nikod\'{y}m
derivative of \(\nu\) with respect to \(\mu\). Now fix \(g\in\LIP_{bs}(\M)\)
and \(f\in C^1_c(\M)\cap\LIP(\M)\). We call \(\d f\colon\T\M\to\R\)
the smooth map \((x,v)\mapsto\d_x f[v]\). Therefore
\[\begin{split}
\int g\,\b(f)\,\d\mu&\overset{\eqref{eq:superposition_principle_1}}=
\int\!\!\!\int_0^1 g(\gamma_t)\,(f\circ\gamma)'_t\,\d t\,\d\ppi(\gamma)
=\int g\big({\rm e}(\gamma,t)\big)\,\d f
\big({\sf Der}(\gamma,t)\big)\,\d\hat\ppi(\gamma,t)\\
&\overset{\eqref{eq:disint_2}}=
\int\bigg(\int g\big({\rm e}(\gamma,t)\big)\,\d f
\big({\sf Der}(\gamma,t)\big)\,\d\hat\ppi_x(\gamma,t)\bigg)\,\d\nu(x)\\
&\overset{\phantom{\eqref{eq:superposition_principle_2}}}=
\int g(x)\,\bigg(\int_{\mathcal D_x}\d_x f[v]\,\d\mathfrak n_x(v)\bigg)\,\d\nu(x).
\end{split}\]
Since \(g\in\LIP_{bs}(\M)\) is arbitrary, we deduce that
\(\b(f)(x)=\frac{\d\nu}{\d\mu}(x)\int_{\mathcal D_x}\d_x f[v]\,\d\mathfrak n_x(v)\)
holds for \(\mu\)-a.e.\ point \(x\in\M\). Being the map
\(\d_x f|_{\mathcal D_x}\colon\mathcal D_x\to\R\) linear and continuous,
we conclude that
\[\begin{split}
\b(f)(x)&=\frac{\d\nu}{\d\mu}(x)\int_{\mathcal D_x}\d_x f[v]\,\d\mathfrak n_x(v)
=\d_x f\bigg[\frac{\d\nu}{\d\mu}(x)\int_{\mathcal D_x}v\,\d\mathfrak n_x(v)\bigg]
=\d_x f\big[{\rm I}(\b)(x)\big]\\
&=\d_\H f(x)\big[{\rm I}(\b)(x)\big]
\end{split}\]
is satisfied for \(\mu\)-a.e.\ \(x\in\M\), thus proving \eqref{eq:I(b)}.\\
{\color{blue}\textsc{Isometry.}}
Let \(\b\in{\rm Der}^{2,2}(\M;\mu)\) be fixed.
We deduce from the \(\mu\)-a.e.\ identity
\(|\b|=\Phi\frac{\d\nu}{\d\mu}\) that
\[{\big\|{\rm I}(\b)(x)\big\|}_x=
\frac{\d\nu}{\d\mu}(x)\,{\left\|\int_{\mathcal D_x}v\,\d\mathfrak n_x(v)\right\|}_x
\leq\frac{\d\nu}{\d\mu}(x)\int_{\mathcal D_x}{\|v\|}_x\,\d\mathfrak n_x(v)
=|\b|(x)\quad\text{ for }\mu\text{-a.e.\ }x\in\M.\]
In order to prove the converse inequality, pick a countable family
\(\mathscr F\subseteq C^1_c(\M)\cap\LIP(\M)\) as in Theorem
\ref{thm:ptwse_norm_b_bis}. Therefore, for any \(f\in\mathscr F\) it holds that
\[\b(f)(x)\overset{\eqref{eq:I(b)}}=\d_\H f(x)\big[{\rm I}(\b)(x)\big]
\leq{\big\|\d_\H f(x)\big\|}_x^*\,{\big\|{\rm I}(\b)(x)\big\|}_x
\overset{\eqref{eq:estimate_nabla_H}}\leq\Lip(f)\,{\big\|{\rm I}(\b)(x)\big\|}_x
\leq{\big\|{\rm I}(\b)(x)\big\|}_x\]
for \(\mu\)-a.e.\ \(x\in\M\), whence
\(|\b|(x)=\big({\rm ess\,sup}_{f\in\mathscr F}\,\b(f)\big)(x)
\leq{\big\|{\rm I}(\b)(x)\big\|}_x\) holds for \(\mu\)-a.e.\ point \(x\in\M\).
This completes the proof of \eqref{eq:I_isometry}.
\end{proof}
Finally, we conclude by expounding how to deduce from Theorem
\ref{thm:embedding} that all sub-Riemannian manifolds are universally
infinitesimally Hilbertian. This is the content of the following
result, which has been already stated in Theorem \ref{thm:uiH_intro}.
\begin{theorem}[Infinitesimal Hilbertianity of sub-Riemannian manifolds]\label{thm:uiH}
Let \((\M,\E,\sigma,\psi)\) be a sub-Riemannian manifold with \(\sfd_{\rm CC}\)
complete. Let \(\mu\) be a non-negative Radon measure on \((\M,\sfd_{\rm CC})\).
Then the metric measure space \((\M,\sfd_{\rm CC},\mu)\) is infinitesimally
Hilbertian.
\end{theorem}
\begin{proof}
Let \(\bar x\in{\rm spt}(\mu)\) be fixed.
We define \(B_n\coloneqq\bar B_n(\bar x)\)
and \(\mu_n\coloneqq\mu|_{B_n}\) for every \(n\in\N\).
We know from \cite[Proposition 2.6]{Gigli12} and \cite[Theorem 7.2.5]{DiMarino_thesis}
that for any \(n\in\N\) it holds that
\begin{equation}\label{eq:local_Sob}
f\in W^{1,2}(\M,\sfd_{\rm CC},\mu)\quad\Longrightarrow
\quad f\in W^{1,2}(\M,\sfd_{\rm CC},\mu_n)\;\text{ and }\;
|Df|_{\mu_n}=|Df|_\mu\;\;\mu_n\text{-a.e..}
\end{equation}
This ensures that, in order to prove that \((\M,\sfd_{\rm CC},\mu)\) is
infinitesimally Hilbertian, it is enough to show that \(W^{1,2}(\M,\sfd_{\rm CC},\mu_n)\)
is a Hilbert space for every \(n\in\N\). Given that \(\mu_n\) is a finite measure,
we can apply Theorem \ref{thm:embedding} and Remark \ref{rmk:geom_pr} to deduce that
\[|\b+\b'|^2+|\b-\b'|^2=2\,|\b|^2+2\,|\b'|^2\;\;\;\mu\text{-a.e.}
\quad\text{ for every }\b,\b'\in{\rm Der}^{2,2}(\M;\mu).\]
Hence \((\M,\sfd_{\rm CC},\mu_n)\) is infinitesimally Hilbertian
by Proposition \ref{prop:suff_cond_iH}. The statement is achieved.
\end{proof}
\begin{remark}{\rm
Given a sub-Finsler manifold \((\M,\E,\sigma,\psi)\) equipped with a non-negative
Radon measure \(\mu\), it is not necessarily true that
\(W^{1,2}(\M,\sfd_{\rm CC},\mu)\) is a Hilbert space.
Nevertheless, we can still deduce from Theorem \ref{thm:embedding} that
\(W^{1,2}(\M,\sfd_{\rm CC},\mu)\) is reflexive, as we are going to explain.

First of all, it can be readily checked that \(L^2(\H\M;\mu)\)
is a reflexive Banach space if endowed with the norm
\(L^2(\H\M;\mu)\ni v\mapsto\big(\int\big\|v(x)\big\|_x^2
\,\d\mu(x)\big)^{\nicefrac{1}{2}}\). Calling \(\B\) the dual of
\(\big({\rm Der}^{2,2}(\M;\mu),\|\cdot\|_2\big)\), where the norm
\(\|\cdot\|_2\) is defined as in Remark \ref{rmk:mathscr_L_f}, we
deduce from Theorem \ref{thm:embedding} that \(\B\) is a reflexive Banach space.
Consequently, the product space \(L^2(\mu)\times\B\) is reflexive as well.
Define \(\mathscr L_f\in\B\) for every \(f\in W^{1,2}(\M,\sfd_{\rm CC},\mu)\)
as in \eqref{eq:def_mathscr_L_f}. Observe that Remark \ref{rmk:mathscr_L_f}
grants that the linear operator
\[\begin{split}
W^{1,2}(\M,\sfd_{\rm CC},\mu)&\longrightarrow L^2(\mu)\times\B,\\
f&\longmapsto(f,\mathscr L_f)
\end{split}\]
is an isometry. Therefore, we can finally conclude
that \(W^{1,2}(\M,\sfd_{\rm CC},\mu)\) is reflexive.
\fr}\end{remark}
\def\cprime{$'$} \def\cprime{$'$}

\end{document}